\documentclass[11pt]{birkjour}
\usepackage[table,xcdraw]{xcolor}
\usepackage{tikz}
\usetikzlibrary{arrows.meta, positioning}
\usepackage{tikz-cd}
\usepackage{mathdots}
\usepackage{multicol}
\usepackage{color}
\usepackage{faktor}
\usepackage{textcomp}
\usepackage{tkz-graph}
\usepackage{tikz-cd}
\usetikzlibrary{calc}
\usepackage{tikz}
\usetikzlibrary{arrows}
\usepackage{stackrel}
\usepackage{faktor}
\usepackage{stackengine}
\usepackage{amssymb}
\usepackage[mathscr]{euscript}
\usepackage{comment}
\usepackage[skip=0.4\baselineskip]{caption}
\usepackage{hyperref}
\stackMath

 \newtheorem{thm}{Theorem}[section]
 \newtheorem{cor}[thm]{Corollary}
 \newtheorem{lem}[thm]{Lemma}
 \newtheorem{prop}[thm]{Proposition}
 \theoremstyle{definition}
 
 \theoremstyle{remark}
 \newtheorem{rem}[thm]{Remark}
 
 \numberwithin{equation}{section}

\begin{document}

%-------------------------------------------------------------------------
% editorial commands: to be inserted by the editorial office
%
%\firstpage{1} \volume{228} \Copyrightyear{2004} \DOI{003-0001}
%
%
%\seriesextra{Just an add-on}
%\seriesextraline{This is the Concrete Title of this Book\br H.E. R and S.T.C. W, Eds.}
%
% for journals:
%
%\firstpage{1}
%\issuenumber{1}
%\Volumeandyear{1 (2004)}
%\Copyrightyear{2004}
%\DOI{003-xxxx-y}
%\Signet
%\commby{inhouse}
%\submitted{March 14, 2003}
%\received{March 16, 2000}
%\revised{June 1, 2000}
%\accepted{July 22, 2000}
%
%
%
%---------------------------------------------------------------------------
%Insert here the title, affiliations and abstract:
%
\title[On the Number of Exceptional pairs over a Class of Nakayama Algebras]
{On the Number of Exceptional Pairs Over a Class of Nakayama Algebras}
%----------Author 1
\author[PFFE]{Pedro Fernando Fern\'andez Espinosa}
\address{
	Escuela de Matemáticas y Estadistica  \\
	Universidad Pedagógica y Tecnológica de Colombia \\
	Carrera 18 con Calle 22\\
	Duitama - Boyacá\\
    ORCID ID: 0000-0002-2650-4536}
\email{pedro.fernandez01@uptc.edu.co}
\thanks{This work was completed with the support of Universidad Pedagógica y Tecnológica de Colombia and Universidad de Antioquia.}
%----------Author 2
\author[DRM]{David Reynoso--Mercado}
\address{Instituto de Matemáticas\br
	Universidad de Antioquia \br
	Calle 67 No. 53-108\br
	Medell\'in-Colombia\\
    ORCID ID: 0009-0008-8299-3486}
\email{david.reynoso@udea.edu.co}
%----------Author 3
%----------classification, keywords, date
\subjclass{Primary 16G20; 16G70;  Secondary 05A17; 16E05}
\keywords{Exceptional module, exceptional pair, Auslander- Reiten quiver, Nakayama Algebras, integer sequences}
\date{November 03, 2025}
%----------additions
%\dedicatory{Dedicated to Professor Agust\'in Moreno Ca\~nadas on the Occasion of his 60th Birthday.}
%%% ----------------------------------------------------------------------
\begin{abstract}
In this paper some combinatorial and homological tools are used to describe and give an explicit formula for the number of exceptional pairs (exceptional sequences of length two) for some classes of Nakayama Algebras and for the Auslander algebra of a radical square zero algebra of type $\mathbb{A}_n$. In addition, we explore how the number of exceptional pairs can be connected with some integer sequences in the OEIS (The On-line Encyclopedia of Integer Sequences).
\end{abstract}
%%% ----------------------------------------------------------------------
\maketitle
%%% ----------------------------------------------------------------------
%\tableofcontents
\section{Introduction}
Exceptional sequences were introduced by Gorodentsev and Rudakov in \cite{gorodentsev1987}  to study exceptional vector bundles on $\mathbb{P}^2$, in fact, exceptional sequences can be regarded as a certain variation of tilting objects and are useful for studying the structure of derived
categories of algebraic varieties (see \cite{bondal1989}, \cite{rudakov1990}). \par\smallskip

\noindent  In the context of finite-dimensional algebras  Crawley-Boevey in \cite{crawley1992} and C.M. Ringel in \cite{ringel1994} defined exceptional sequences for path algebras of acyclic quivers as sequences of representations of a given length satisfying specific homological properties.  Recently, exceptional sequences have been studied in various contexts, for instance, in 2010, K. Igusa and R. Schiffler in \cite{igusa2010} established interesting connections between exceptional sequences and cluster algebras. In 2017, H. Krause \cite{krause2017} presented a close relationship between exceptional sequences and highest weight categories, also H. Krause and A. Hubery \cite{krause2013} found out an interest connection between bilinear lattices and exceptional sequences. Furthermore, in 2018 A. Buan and B. Marsh \cite{buan2018} 
introduced the notions of $\tau$-exceptional and signed $\tau$-exceptional sequences for any finite dimensional algebra; in the hereditary case, these notions coincide with the classical exceptional and signed exceptional sequences. In addition, there is a significant amount of work on exceptional sequences and their generalizations from a combinatorial point of view  \cite{araya2009, Carrick, igusa1,igusa2010, igusa3, igusa4, Msapato,ringel2013, Seidel2001, Sen1, Sen2} among others. It is important to emphasize that, for Dynkin algebras, there are several approaches—using different tools and perspectives—to obtain the number of exceptional pairs and exceptional sequences (see \cite{araya2009, Carrick, ringel2013}), whereas for bound quiver algebras only a few cases have been studied.
\par\smallskip 
 
\noindent This work focuses on the combinatorial properties of exceptional pairs in the context of algebras $A = \mathbb{F}Q/I$, where $Q$ is a Dynkin quiver of type $\mathbb{A}_n$ with a linear orientation, and $I$ is an admissible ideal generated by a collection of disjoint relations.
As two particular cases, we provide explicit formulas for the number of exceptional pairs when $I$ is generated by a single relation $R$ of length $k$ starting at vertex $1$, and when no admissible ideal $I$ is considered (this case was already studied by O. Bernal in \cite{Bernal}). \par\smallskip 

\noindent Moreover, we examine the case in which $I$ is an admissible ideal generated by all possible relations of length two, and as a direct consequence, we obtain an explicit formula for the Auslander algebra of a radical square zero algebra of type $\mathbb{A}_n$ studied in \cite{Chen}.
In all the aforementioned cases, we establish an algorithm to compute the number of exceptional pairs. \par\smallskip

\noindent Finally, and of equal importance, we introduce several new integer sequences arising from exceptional pairs in the OEIS (The On-Line Encyclopedia of Integer Sequences, see \cite{OEIS}), as well as point out others that are already encoded there.
 
\par\smallskip 

\noindent This paper is organized as follows. In Section \ref{preliminaries}, we recall the main notation and definitions concerning exceptional pairs, Nakayama algebras, and the Auslander algebra of a radical square-zero algebra of type $\mathbb{A}_n$.
\par\smallskip
\noindent In Section \ref{General}, we describe and enumerate the indecomposable modules of any Nakayama algebra and establish conditions that determine whether a pair of $A$-modules $(M, N)$ is exceptional. In particular, we provide criteria to determine when $\mathrm{Hom}(M, N) \cong 0$ and $\mathrm{Ext}^i(M, N) \cong  0$ for all $i \geq 1$. As an application, in Section \ref{particularcases} we compute the number of exceptional pairs when $Q$ is a Dynkin quiver of type $\mathbb{A}_n$, the number of exceptional pairs for algebras of the form $A = \mathbb{F}Q/I$, where $I$ is an admissible ideal generated by a collection of disjoint relations and we provide an explicit formula for the number of exceptional pairs when $I$ is generated by a single relation $R$ of length $k$ starting at vertex $1$. In addition, we describe several interesting integer sequences from the On-Line Encyclopedia of Integer Sequences (OEIS) that arise from these computations.\par\smallskip 

\noindent Finally, in Section \ref{rad2}, we study the case where $I$ is the admissible ideal generated by all possible relations of length two. As a direct consequence, we obtain an explicit formula for the Auslander algebra of a radical square-zero algebra of type $\mathbb{A}_n$, as studied in \cite{Chen}.

\section{Preliminaries}\label{preliminaries}
In this section, we recall main definitions and notation to be used throughout the paper \cite{Happel,Sen2,Chen,crawley1992,ringel2013}. \par\smallskip

\noindent Let $\mathbb{F}$ be an algebraically closed field, throughout this paper we denote by $A$ as a path algebra $\mathbb{F}Q/I$, where $Q$ is a linear oriented Dynkin quiver of type $\mathbb{A}_n$ where $I$ is an admissible ideal generated by a collection of admissible  relations. In fact, we recall that if $Q$ be either a linearly oriented quiver with underlying graph $\mathbb{A}_n$ or a cycle $\widetilde{\mathbb{A}_n}$ with cyclic
orientation. That is, $Q$ is one of the following quivers

\begin{figure}[h!]
\begin{center}
\begin{tikzpicture}[scale=0.55]
\def \radius {2cm}
\def \margin {14} % margin in angles, depends on the radius
{
\node at ({360/8 * (1 - 1)}:\radius)(U2) {\tiny{$1$}};

\node at ({360/8 * (2 - 1)}:\radius)(U3) {\tiny{$2$}};

\node at ({360/8 * (3 - 1)}:\radius) (U4){\tiny{$3$}};

\node at ({360/8 * (4 - 1)}:\radius)(U5) {\tiny{$4$}};

\node at ({360/8 * (5 - 1)}:\radius)(U6) {\tiny{$5$}};

\node at ({360/8 * (6 - 1)}:\radius)(U7) {\tiny{$\ddots$}};

\node at ({360/8 * (7 - 1)}:\radius)(Uk-1) {\tiny{$n-1$}};

\node at ({360/8 * (8 - 1)}:\radius)(Uk) {\tiny{$n$}};

%Arrows in the circle

\draw[->, >=latex] ({360/8* (1 - 1)+\margin}:\radius)
arc ({360/8 * (1 - 1)+\margin}:{360/8 * (1)-\margin}:\radius)node[midway,sloped,above] {};

\draw[->, >=latex] ({360/8* (2 - 1)+\margin}:\radius)
arc ({360/8 * (2 - 1)+\margin}:{360/8 * (2)-\margin}:\radius)node[midway,sloped,above] {};

\draw[->, >=latex] ({360/8* (3 - 1)+\margin}:\radius)
arc ({360/8 * (3 - 1)+\margin}:{360/8 * (3)-\margin}:\radius)node[midway,sloped,above] {};

\draw[->, >=latex] ({360/8* (4 - 1)+\margin}:\radius)
arc ({360/8 * (4 - 1)+\margin}:{360/8 * (4)-\margin}:\radius)node[midway,sloped,above] {};

\draw[->, >=latex] ({360/8* (5 - 1)+\margin}:\radius)
arc ({360/8 * (5 - 1)+\margin}:{360/8 * (5)-\margin}:\radius)node[midway,sloped,below] {};

\draw[->, >=latex] ({360/8* (6 - 1)+\margin}:\radius)
arc ({360/8 * (6 - 1)+\margin}:{360/8 * (6)-\margin}:\radius)node[midway,sloped,below] {};

\draw[->, >=latex] ({360/8* (7 - 1)+\margin}:\radius)
arc ({360/8 * (7 - 1)+\margin}:{360/8 * (7)-\margin}:\radius)node[midway,sloped,below] {};

\draw[->, >=latex] ({360/8* (8 - 1)+\margin}:\radius)
arc ({360/8 * (8 - 1)+\margin}:{360/8 * (8)-\margin}:\radius)node[midway,sloped,below] {};
}

%% Linear
\node["\tiny{$1$}"] at (4,0)(1){$\bullet$};
\node["\tiny{$2$}"] at (5.5,0)(2){$\bullet$};

\node at (6.5,0)(dot){\tiny{$\cdots$}};

\node["\tiny{$n-1$}"] at (7.5,0)(n-1){$\bullet$};
\node["\tiny{$n$}"] at (9,0)(n){$\bullet$};

\node at (3,0)(or){\small{or}};

\path[->] (1) edge (2);
\path[->] (n-1) edge (n);

\end{tikzpicture}
\caption{Quiver $\widetilde{\mathbb{A}_n}$ with cyclic orientation and Dynkin diagram $\mathbb{A}_n$ linearly oriented.}
\end{center}
\end{figure}
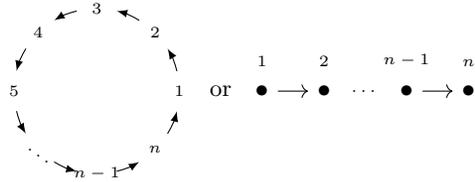\label{Nakayama}

\noindent Then $A$ is called a Nakayama algebra \cite{Happel}. In this paper, we fix $n \ge 3$ and restrict our attention to Nakayama algebras arising from linearly oriented quivers with underlying graph $\mathbb{A}_n$. Throughout this paper, the term Nakayama algebra of type I will refer exclusively to algebras of this form. In particular, we analyze the following forms of the admissible ideal $I$ generated by 

\begin{itemize}
\item Any ideal generated by a collection of overlap-free relations (see Subsection \ref{nointer}).
    \item A single relation $R=\alpha_1\cdots\alpha_{1+p}$ of length $m=1+p$ starting at vertex $1$ (see Corollary \ref{cor1}).
    \item All possible relations of length two, this case is also known as radical square zero Nakayama algebra (see 
    Subsection \ref{rad2}).
\end{itemize}

\noindent Also, we point out that according to \cite{Chen}  $A$ is the Auslander algebra of a radical square zero algebra of type $\mathbb{A}_n$.
Then $A$ is given by the following quiver $Q$ :
$$
1 \stackrel{a_1}{\rightarrow} 2 \stackrel{a_2}{\rightarrow} \cdots \stackrel{a_{2 m-3}}{\rightarrow} 2 m-2 \stackrel{a_{2 m-2}}{\rightarrow} 2 m-1
$$
\noindent with the relations:
$$
a_{2 k-1} a_{2 k}=0 \quad (1 \leq k \leq m-1)
$$

\noindent Another key concept for this paper is the concept of exceptional pair. A pair of modules $(M, N)$ over $A$ is called an exceptional pair if $\operatorname{Hom}(N, M)\cong 0$ and if $\operatorname{Ext}^i(N, M)\cong0,\, i \geq 1$. A sequence of  modules $\left(M_1, \ldots, M_t\right)$ is called an exceptional sequence if for all $1 \leq i<j \leq t$, $\left(M_i, M_j\right)$ is an exceptional pair. Moreover it is called complete if $t=n$. Here we want to emphasize that elements of exceptional sequences are isomorphism classes of $A$ modules. The study of exceptional sequences, both in the context of coherent sheaves and their derived categories, as well as in quiver representations and their combinatorics, plays a fundamental role. Enumerative aspects are also of particular interest. For instance, the number of complete exceptional sequences of hereditary algebras of types $\mathbb{A}_n$ and $\mathbb{D}_n$ are given by $(n+1)^{n-1}$ and $2(n-1)^n$ (See \cite{Sen2, Seidel2001}). This work contributes to the study of combinatorial and enumerative aspects of these algebraic objects in the context of Nakayama algebras.

\section{General Results: Any Nakayama Algebra of type I}\label{General}

In this section, we present some of the main results of this paper. In particular, we describe and enumerate all indecomposable modules of an arbitrary Nakayama algebra by using the combinatorial structure of intervals (see Theorem \ref{thmind}). Moreover, for a fixed indecomposable module \(N = [a,b]\), in Proposition \ref{prophomab} we compute the number of modules \(M\) such that \(\mathrm{Hom}(M,N) \cong\mathbb{F}\). Furthermore, in Proposition \ref{cor3.9}, by considering the structure of the projective resolution of \([a,b]\), we determine how many modules \(M\) satisfy \(\mathrm{Ext}^i(N,M) \cong \mathbb{F}\) for some \(i \geq 1\). Finally, we introduce an explicit algorithm to compute the number of exceptional pairs by combining Theorem \ref{thmind}, Proposition \ref{prophomab}, and Proposition \ref{cor3.9}.

\subsection{Number of indecomposable modules}

In this section, we describe and enumerate the indecomposable modules of any Nakayama algebra of type I. To do so, we first point out that any Nakayama algebra can be completely described by the following data:

\begin{itemize}
    \item [a.] The number of vertices of the quiver $n$.
    \item [b.] A list of $k$ relations in the quiver. Each relation is of the form $R_l= \alpha_{i_l} \alpha_{i_l+1} \cdots \alpha_{i_l+p_l}$ with $1\leq l\leq k$, $1<p_l \leq n-i_l$, where $i_l$ is the vertex where the relation begins and $p_l$ is its corresponding length.
\end{itemize}

\noindent From now on, we use intervals \([i,j]\) with \(1 \leq i \leq j \leq n\), to describe the indecomposable representations. We mean by \([i,j]\) with \(1 \leq i \leq j \leq n\)  the representation in which the vector space at vertex \(l\) is the field \(\mathbb{F}\) if \(i \leq l \leq j\) and \(0\) at every other vertex. That is,
$$0 \longrightarrow \cdots \longrightarrow \overset{i}{\mathbb{F}} \longrightarrow \mathbb{F} \longrightarrow \cdots \longrightarrow \overset{j}{\mathbb{F}} \longrightarrow 0 \longrightarrow \cdots \longrightarrow 0,$$
with the identity maps on the nonzero components (and the zero maps elsewhere).

\noindent It is straightforward to verify that for each vertex \(i \in {1,\dots,n}\):
\begin{itemize}
    \item If \(i_{j-1} < i \leq i_j\), the corresponding projective module is
\begin{eqnarray}
    P(i)&=&[i,i_j+p_j]
\end{eqnarray}
\item To determine the corresponding injective module, assume \(i_{j-1} + p_{j-1} < i \leq i_j + p_j\). Then
\begin{eqnarray}
        I(i)&=&[i_j+1,i]
\end{eqnarray}
 \item  Finally, the corresponding simple modules are
\begin{eqnarray}
    S(i)&=&[i,i]
\end{eqnarray}
\end{itemize}
\noindent Here, \([a,b]\) denotes the interval representation defined previously.

\noindent Note that the projective–injective modules correspond precisely to the intervals $I_1,\cdots,\,I_{k+1}$; in fact,
$$P(i_{j-1}+1)=I_j=I(i_j+p_j)\qquad\text{for every } j\in{1,\dots,k+1}.$$

\noindent In general, if \([a,b]\) and \([a',b']\) are two intervals representing some module, we say that they intersect if and only if $a\leq b' ;\text{ and }; a'\leq b$. When \([a,b]\) and \([a',b']\) intersect, we denote their intersection by
$$[a,b]\cap[a',b']=[c,d],$$
where \(c=\max\{a,a'\}\) and \(d=\min\{b,b'\}\).

\noindent  Following the knitting algorithm for the Auslander–Reiten quiver (see \cite{schiffler2014}) , we obtain that for each \(l \in {1, \dots, n}\), if \(i_{j-1} < l \leq i_j\), then an interval \([a,b]\) represents a module if and only if \(a = l\) and \(b \leq i_j + p_j\). In this way, all modules are determined by the projective modules. For each \(l \in {1, \dots, n-1}\), we have the sets
$$A_{P(l)}=\{[l,l],\,[l,l+1],\cdots,\,[l,i_j+p_j]\},$$
and \(A_{P(n)} = {[n,n]}\). Moreover, for each \(P(l)\), there are precisely \(|A_{P(l)}|\) modules, and it is not difficult to verify that
$$|A_{P(l)}|=i_j+p_j-l+1 \qquad \text{whenever } i_{j-1} < l \leq i_j.$$

\noindent Let $A = \mathbb{F}Q/I$ be a Nakayama algebra. We denoted by $\mathrm{Ind} (A)$ the set of indecomposable $A$-modules.

\begin{rem}
The family \(\{A_{P(i)} \mid 1 \le i \le n\}\) forms a partition of \(\mathrm{Ind}(A)\).
\end{rem}

\begin{thm}\label{thmind}
Let $Q$  be a linearly oriented quiver of type $\mathbb{A}_n$ and let  $A = \mathbb{F}Q/I$ be a Nakayama algebra, where $I$ is generated by minimal relations of the form
$\alpha_{i_l}\alpha_{i_l+1}\cdots\alpha_{i_l + p_l}$, $1\le l\le k.$
 Then the number of indecomposable $A$-modules, denoted $|\mathrm{Ind} (A)|$, is:
\[
|\mathrm{Ind} (A)| = \sum_{l=1}^{k+1} (i_l - i_{l-1})(i_l + p_l) - \frac{n(n-1)}{2},
\]
with the conventions  $i_0=p_{k+1} = 0$ y $i_{k+1} = n$.
\end{thm}
\begin{proof}
As we saw before the family $\{A_{P(i)} \mid 1 \le i \le n\}$ induce a partition of \(\mathrm{Ind}(A)\), then
\begin{eqnarray*}
\left|\textrm{Ind }(A)\right|&=&\sum_{i=1}^{n}|A_{P(i)}|\\
&=&\sum_{l=1}^{k+1}\left(\sum_{i=i_{l-1}+1}^{i_l}i_l+p_l-i+1\right)\\
&=&\sum_{l=1}^{k+1}\left((i_l+p_l+1)(i_l-i_{l-1})-\sum_{i=i_{l-1}+1}^{i_l}i\right)\\
&=&\sum_{l=1}^{k+1}(i_l-i_{l-1})(i_l+p_l)+\sum_{l=1}^{k+1}(i_l-i_{l-1})-\sum_{l=1}^{k+1}\sum_{i=i_{l-1}+1}^{i_l}i\\
&=&\sum_{l=1}^{k+1}(i_l-i_{l-1})(i_l+p_l)+(i_{k+1}-i_0)-\sum_{l=1}^ni\\
&=&\sum_{l=1}^{k+1}(i_l-i_{l-1})(i_l+p_l)+(n-\frac{n(n+1)}{2})\\
&=&\sum_{l=1}^{k+1}(i_l-i_{l-1})(i_l+p_l)-\frac{n(n-1)}{2} 
\end{eqnarray*}
\end{proof}

%%% Page 2 

\subsection{On the number of pairs  \(\mathrm{Hom}(M,N) \cong\mathbb{F}\)}
\noindent  In this section, for a fixed indecomposable module \(N = [a,b]\) in any Nakayama algebra of type I, we compute the number of modules \(M\) such that \(\mathrm{Hom}(M,N) \cong\mathbb{F}\). Firstly, Proposition \ref{prophom} gives a criterion to establish when $\operatorname{Hom}(M,N) \cong 0$.

\begin{prop}\label{prophom}
    Let \(M\) and \(N\) be indecomposable modules, and let \([a,b]\) and \([a',b']\) be the interval modules associated with their representations. Then \(\operatorname{Hom}(M,N) \cong 0\) if and only if one of the following conditions holds:
\begin{enumerate}
  \item[(F1)] \([a,b]\) and \([a',b']\) do not intersect.
  \item[(F2)] \(b < b'\) or \(a < a'\).
\end{enumerate}
\end{prop}
\begin{proof}
It is straightforward to verify that if the intervals satisfy \textnormal{(F1)} or \textnormal{(F2)}, then \(\operatorname{Hom}(M,N) = 0\).
Now assume that the intervals satisfy \(a' \leq a \leq b' \leq b\), that is, neither \textnormal{(F1)} nor \textnormal{(F2)} holds. In this situation, it is not difficult to show that the map \(f = (f_i)_{1 \leq i \leq n}\), defined by \(f_i = 0\) for \(i < a\) or \(i > b'\) and \(f_i = 1_{\mathbb{F}}\) for \(a \leq i \leq b'\), is a nonzero morphism from \(M\) to \(N\). Hence \(\operatorname{Hom}(M,N) \cong\mathbb{F}\).
\end{proof}

\noindent From Proposition \ref{prophom} we obtain the following observation.

\begin{rem}\label{remhom1}
\(\operatorname{Hom}(M,N) \cong \mathbb{F}\) if and only if \(a' \le a \le b' \le b\).
%zEquivalently, \(\dim_{\mathbb{F}} \operatorname{Hom}(A,A') = 1\) precisely when \(a' \le a \le b' \le b\).
\end{rem}

\begin{prop}\label{prophomab}
Let \([a,b]\) be the interval associated with the representation of an indecomposable \(A\)-module. 
Denote by \(H_{[a,b]}\) the number of modules \(M\) such that \(\operatorname{Hom}(M,[a,b]) \cong \mathbb{F}\).
Then
\begin{eqnarray}\label{car_hom}
H_{[a,b]} =\displaystyle \sum_{m = a}^{b} (i_{l_m} + p_{l_m}) + (b - a + 1)(1 - b),
\end{eqnarray}
where for each \(m \in \{1, \dots, n\}\) the index \(l_m \in \{1, \dots, k\}\) is defined by \(i_{l_m - 1} < m \le i_{l_m}\). In particular if $l_a=l_b$, then $H_{[a,b]}=(b - a + 1)\bigl(i_{l_a} + p_{l_a} + 1 - b\bigr)$.
\end{prop}

\begin{proof}
Let \([a,b]\) be the interval associated with the representation of an indecomposable \(A\)-module. 
Define the indicator function
\[
f_{[a,b]} : \mathrm{Ind}(A) \longrightarrow \{0,1\}, \qquad
f_{[a,b]}([i,j]) \;=\;
\begin{cases}
1, & \text{if } \operatorname{Hom}([i,j],[a,b]) \cong \mathbb{F},\\[2pt]
0, & \text{if } \operatorname{Hom}([i,j],[a,b]) = 0.
\end{cases}
\]
Since \(\{A_{P(i)}\}_{i=1}^n\) is a partition of \(\mathrm{Ind}(A)\), we have
\[
H_{[a,b]}
\;:=\;
\sum_{i=1}^n \;\sum_{[i,j]\in A_{P(i)}} f_{[a,b]}([i,j]),
\]
which ranges over all intervals corresponding to the indecomposable modules in \(\mathrm{Ind}(A)\). 
Thus \(H_{[a,b]}\) equals the number of modules \(M\) such that \(\operatorname{Hom}(M,[a,b]) \cong \mathbb{F}\).

\medskip
\noindent By Proposition~\ref{prophom}, \(f_{[a,b]}([i,j])=0\) whenever \(i<a\) or \(i>b\); hence
\[
H_{[a,b]}
\;=\;
\sum_{i=a}^{n} \;\sum_{[i,j]\in A_{P(i)}} f_{[a,b]}([i,j]).
\]
Moreover, for \(i\ge b+1\) the intervals \([i,j]\) and \([a,b]\) do not intersect, so 
\(f_{[a,b]}([i,j])=0\) in this range and therefore
\[
H_{[a,b]}
\;=\;
\sum_{i=a}^{b} \;\sum_{[i,j]\in A_{P(i)}} f_{[a,b]}([i,j]).
\]
Now, let \(l_m \in \{1,\dots,k\}\) be determined by \(i_{l_m-1} < m \le i_{l_m}\).
Since \(f_{[a,b]}([i,j])=0\) for \(j<b\), the only relevant elements of each \(A_{P(i)}\) are
\(\{[i,b], [i,b+1], \dots, [i,\, i_{l_i}+p_{l_i}]\}\). Hence
\[
\begin{aligned}
H_{[a,b]}
&= \sum_{i=a}^{b} \;\sum_{j=b}^{\,i_{l_i}+p_{l_i}} f_{[a,b]}([i,j])
 = \sum_{m=a}^{b} \bigl(i_{l_m}+p_{l_m}-b+1\bigr)
%\\[2pt]
= \sum_{m=a}^{b} (i_{l_m}+p_{l_m}) \;+\; (b-a+1)(1-b).
\end{aligned}
\]
\end{proof}

\subsection{On the number of pairs  \(\mathrm{Ext}(M,N) \cong\mathbb{F}\)}

\noindent  In this section, for two indecomposable modules $M$ and $N$, we are interested in determining when
\[
\operatorname{Ext}^i(M,N) \cong \mathbb{F},
\]
and in enumerating such cases. \par\smallskip

\noindent Since \(A\) arises from a Nakayama algebra of type I, every module admits a finite projective resolution. If \(M\) has a projective resolution
\[
0 \longrightarrow P_m \xrightarrow{f_m} P_{m-1} \longrightarrow \cdots \longrightarrow
P_1 \xrightarrow{f_1} P_0 \xrightarrow{f_0} M \longrightarrow 0,
\]
\(\operatorname{Ext}^i(M,N) \cong \mathbb{F}\) if and only if the following conditions hold for some \(i \in \{1,\dots,m\}\):
\[
\begin{cases}
\operatorname{Hom}(P_{i+1},N) \cong 0 \;\cong\; \operatorname{Hom}(P_{i-1},N)
\ \text{and}\ \operatorname{Hom}(P_i,N) \cong \mathbb{F}, & \text{if } 1 \le i \le m-1,\\[4pt]
\operatorname{Hom}(P_{i-1},N) \cong 0
\ \text{and}\ \operatorname{Hom}(P_i,N) \cong \mathbb{F}, & \text{if } i = m.
\end{cases}
\]

\noindent The following result uses the combinatorial description of intervals to enumerate the pairs for which $\operatorname{Ext}^i(M,N) \cong \mathbb{F}$.

\begin{prop}\label{prop3.8}
Let $a_1 \le a_2 \le b_1 \le a_3 \le b_2 \le b_3$ be natural numbers. Then the following statements hold:
\begin{enumerate}
    \item The number of intervals $[c,d]$ that simultaneously satisfy
    \[
    \operatorname{Hom}([a_1,b_1],[c,d]) \cong 0,\qquad
    \operatorname{Hom}([a_2,b_2],[c,d]) \cong \mathbb{F},\qquad
    \operatorname{Hom}([a_3,b_3],[c,d]) \cong 0
    \]
    is $\,a_1\bigl(a_2-(b_1+1)\bigr) + a_2(a_3-a_2)\,$.
    \item The number of intervals $[c,d]$ that simultaneously satisfy
    \[
    \operatorname{Hom}([a_1,b_1],[c,d]) \cong 0,\qquad
    \operatorname{Hom}([a_2,b_1],[c,d]) \cong \mathbb{F}
    \]
    is $\, (a_2-a_1)\bigl(b_1+1-a_2\bigr)\,$.
\end{enumerate}
\end{prop}

\begin{proof}
On the one hand, by Proposition \ref{prophom}, we have 
$\operatorname{Hom}([a_i,b_i],[c,d]) \cong 0$ whenever at least one of the following holds:
\[
a_i < c, \qquad d < a_i, \qquad b_i < d, \qquad \text{or }\, b_i < c,
\]
with $i = 1$ or $i = 3$.

\noindent On the other hand, by Remark \ref{remhom1}, 
$\operatorname{Hom}([a_2,b_2],[c,d]) \cong \mathbb{F}$ if and only if
\[
c \le a_2 \le d \le b_2.
\]

\noindent By hypothesis, \(a_1 \le a_2 \le b_1 \le a_3 \le b_2 \le b_3\),
and under the constraint \(c \le a_2 \le d \le b_2\),
several of the above inequalities are impossible, namely:
\(d < a_1\), \(b_1 < c\), \(b_3 < d\), \(b_3 < c\), and \(a_3 < c\).

\noindent Therefore, for \(\operatorname{Hom}([a_1,b_1],[c,d])\cong 0\) one must have
\(a_1<c\) or \(b_1<d\); and for \(\operatorname{Hom}([a_3,b_3],[c,d])\cong 0\) the only
remaining possibility is \(d<a_3\).

\smallskip
\noindent We consider two cases.

\paragraph{Case I} \(a_1<c\le a_2\le d<a_3\).
Here \(c\) can be chosen in \(\{a_1+1,\dots,a_2\}\), giving \(a_2-a_1\) choices,
and \(d\) can be chosen in \(\{a_2,\dots,a_3-1\}\), giving \(a_3-a_2\) choices.
Thus this case contributes \((a_2-a_1)(a_3-a_2)\) intervals.

\paragraph{Case II} \(1\le c\le a_2\) and \(b_1+1\le d<a_3\).
Intervals with \(c\in\{a_1+1,\dots,a_2\}\) were already counted in Case~I, so we only need to
consider \(c\in\{1,\dots,a_1\}\). There are \(a_1\) choices for \(c\) and
\(a_3-(b_1+1)\) choices for \(d\), contributing \(a_1\bigl(a_3-(b_1+1)\bigr)\) intervals.

\smallskip
\noindent Summing the contributions of the two cases, the total number of intervals \([c,d]\) satisfying
\[
\operatorname{Hom}([a_1,b_1],[c,d])\cong 0,\qquad
\operatorname{Hom}([a_2,b_2],[c,d])\cong \mathbb{F},\qquad
\operatorname{Hom}([a_3,b_3],[c,d])\cong 0
\]
are 
\[
(a_2-a_1)(a_3-a_2)+a_1\bigl(a_3-(b_1+1)\bigr)
= a_1\bigl(a_2-(b_1+1)\bigr)+a_2(a_3-a_2).
\]

\noindent For item~(2), we require
\(\operatorname{Hom}([a_2,b_1],[c,d])\cong \mathbb{F}\), i.e., \(c\le a_2\le d\le b_1\).
As before, \(b_1<c\), \(d<a_1\), and \(b_1<d\) are impossible under these bounds, hence
the only way to ensure \(\operatorname{Hom}([a_1,b_1],[c,d])\cong 0\) is \(a_1<c\).
Thus \(c\in\{a_1+1,\dots,a_2\}\) (giving \(a_2-a_1\) choices) and
\(d\in\{a_2,\dots,b_1\}\) (giving \(b_1+1-a_2\) choices), for a total of
\((a_2-a_1)(b_1+1-a_2)\) intervals.

\end{proof}

\noindent From the previous results we derive the following two statements, which are crucial to enumerate the interval which satisfies
\(\operatorname{Ext}^i([a,b],[c,d]) \cong \mathbb{F}\). The first shows that there is no double counting, while the second provides the exact number of intervals \([c,d]\) satisfying the above condition for a fixed interval \([a,b]\).

\begin{prop}\label{prop_extji}
Let \(M\) be an indecomposable module with associated interval \([a,b]\). If for some module represented by \([c,d]\) and some \(i \ge 1\) we have
\(\operatorname{Ext}^i([a,b],[c,d]) \cong \mathbb{F}\),
then \(\operatorname{Ext}^j([a,b],[c,d]) \cong 0\) for every \(j \ne i\).
\end{prop}
\begin{proof}
Let \([a,b]\) be the interval associated with an indecomposable \(A\)-module \(M\), and suppose its projective resolution is
\[
0 \longrightarrow [a_m,b_m] \xrightarrow{f_m} [a_{m-1},b_{m-1}] \longrightarrow \cdots \longrightarrow
[a_1,b_1] \xrightarrow{f_1} [a_0,b_0] \xrightarrow{f_0} [a,b] \longrightarrow 0.
\]
In general, by the combinatorial structure of projective resolutions in this setting, we have the chain of inequalities
\[
a_0< a_1< b_0< a_2<\cdots< a_i< b_{i-1}< a_{i+1}< b_i<\cdots< a_m< b_{m-1}=b_m.
\]

\noindent Let \([c,d]\) be a module such that \(\operatorname{Ext}^i([a,b],[c,d])\cong \mathbb{F}\).
If \(i<m\), then by the counting criterion established in Proposition~\ref{prop3.8}, this is equivalent to one of the following two configurations:
\[
\text{(I)}\quad a_{i-1} < c \le a_i \le d < a_{i+1},
\qquad\text{or}\qquad
\text{(II)}\quad 1\le c \le a_i,\; b_{i-1}+1 \le d < a_{i+1}.
\]

\noindent Fix \(j<i\). Then \(a_j<a_i\), hence \(a_{j+1}\le a_i\).
From (I)–(II) we have \(a_i\le d < a_{i+1}\) (in case (II), also \(a_i<b_{i-1}+1\le d < a_{i+1}\)), so in either case \(a_{j+1}\le d\).
Therefore the patterns that would force \(\operatorname{Ext}^j([a,b],[c,d])\cong \mathbb{F}\) (namely \(a_j\le d < a_{j+1}\) or \(a_j<b_{j-1}+1\le d < a_{j+1}\)) cannot occur, and hence
\(\operatorname{Ext}^j([a,b],[c,d])\not\cong \mathbb{F}\).

\noindent If \(j>i\), then \(a_{i+1}\le a_j\), so from (I)–(II) we have \(d < a_{i+1}\le a_j < b_{j-1}+1\).
Thus the configurations required for \(\operatorname{Ext}^j([a,b],[c,d])\cong \mathbb{F}\) again fail, and we conclude
\(\operatorname{Ext}^j([a,b],[c,d])\not\cong \mathbb{F}\).

\noindent If \(i=m\), then by Proposition~\ref{prop3.8} we have \(a_{m-1}< c \le a_m \le d \le b_m\).
For any \(j<m\) we have \(a_{j+1}\le a_m\le d\), so the necessary inequalities for \(\operatorname{Ext}^j([a,b],[c,d])\cong \mathbb{F}\) cannot be satisfied. Hence
\(\operatorname{Ext}^j([a,b],[c,d])\not\cong \mathbb{F}\) for all \(j\ne i\).
\par\bigskip 
\noindent Combining the three paragraphs above, we conclude that if \(\operatorname{Ext}^i([a,b],[c,d])\cong \mathbb{F}\) for some \(i\ge 1\), then
\(\operatorname{Ext}^j([a,b],[c,d])\cong 0\) for every \(j\ne i\).
\end{proof}

\begin{prop}\label{cor3.9}
Let \(M\) be an indecomposable module with associated interval \([a,b]\), with projective resolution given by 
\[
0 \longrightarrow [a_m,b_m] \xrightarrow{f_m} [a_{m-1},b_{m-1}] \longrightarrow \cdots \longrightarrow
[a_1,b_1] \xrightarrow{f_1} [a_0,b_0] \xrightarrow{f_0} [a,b] \longrightarrow 0.
\]
Then the number of interval modules \([c,d]\) such that \(\operatorname{Ext}^i([a,b],[c,d]) \cong \mathbb{F}\) for some \(i \ge 1\) is
\[
a_m(b_m+1)\;-\;a(b+1)\;+\;\sum_{i=1}^{m}\Bigl( 2a_i a_{i-1}\;-\;a_i^{2}\;-\;a_{i-1}(b_{i-1}+1) \Bigr).
\]
\end{prop}

\begin{proof}
The result follows by showing that
\[
E^i_{[a,b]} \;=\;
\begin{cases}
a\bigl(a_1-(b+1)\bigr), & \text{if } i=0,\\[4pt]
a_{i-1}\bigl(a_i-(b_{i-1}+1)\bigr) \;+\; a_i(a_{i+1}-a_i), & \text{if } 1\le i\le m-1,\\[4pt]
(a_m-a_{m-1})\bigl(b_m+1-a_m\bigr), & \text{if } i=m,
\end{cases}
\]
where \(E^i_{[a,b]}\) denotes the number of interval modules \([c,d]\) such that
\(\operatorname{Ext}^i([a,b],[c,d])\cong \mathbb{F}\), and where we use that \(a=a_0\) and \(b_m=b_{m-1}\).\par\smallskip

\noindent Indeed, for \(i=0\) we have
\(\operatorname{Ext}^0([a,b],[c,d])\cong \ker(f_1^*)/\operatorname{im}(f_0^*)\).
By Remark~\ref{remhom1} and Proposition~\ref{prop3.8}, this occurs precisely when
\[
\operatorname{Hom}([a,b],[c,d])\cong 0,\quad
\operatorname{Hom}([a_0,b_0],[c,d])\cong \mathbb{F},\quad
\operatorname{Hom}([a_1,b_1],[c,d])\cong 0,
\]
and the number of such intervals equals
\[
a\bigl(a_0-(b+1)\bigr)+a_0(a_1-a)\;=\;a_0a_1-a(b+1).
\]
Since \([a_0,b_0]\) is the projective cover of \([a,b]\), we have \(a_0=a\), hence the total is
\(a\bigl(a_1-(b+1)\bigr)\).

\noindent For \(1\le i\le m-1\), we have
\(\operatorname{Ext}^i([a,b],[c,d])\cong \ker(f_{i+1}^*)/\operatorname{im}(f_i^*)\),
which holds exactly when
\[
\operatorname{Hom}([a_{i-1},b_{i-1}],[c,d])\cong 0,\quad
\operatorname{Hom}([a_i,b_i],[c,d])\cong \mathbb{F},\quad
\operatorname{Hom}([a_{i+1},b_{i+1}],[c,d])\cong 0.
\]
By Proposition~\ref{prop3.8} the number of such intervals is
\[
a_{i-1}\bigl(a_i-(b_{i-1}+1)\bigr) \;+\; a_i(a_{i+1}-a_i).
\]

\noindent Finally, for \(i=m\) we have \(\ker(f_{m+1}^*)\cong\operatorname{Hom}([a_m,b_m],[c,d])\) since \(f_{m+1}^*=0\).
Moreover, \([a_m,b_m]\) is the projective cover of \(\ker(f_{m-1})\cong [\,b_{m-2}+1,\,b_{m-1}\,]\), and for the resolution to terminate \(\ker(f_{m-1})\) must be projective; hence
\([a_m,b_m]=[\,b_{m-2}+1,\,b_{m-1}\,]\) and thus \(b_m=b_{m-1}\).
Therefore \(\operatorname{Ext}^m([a,b],[c,d])\cong \mathbb{F}\) precisely when
\[
\operatorname{Hom}([a_{m-1},b_m],[c,d])\cong 0
\quad\text{and}\quad
\operatorname{Hom}([a_m,b_m],[c,d])\cong \mathbb{F},
\]
and by Proposition~\ref{prop3.8} the number of such intervals is
\((a_m-a_{m-1})(b_m+1-a_m)\).
\end{proof}

\subsection{Algorithm to compute the number of exceptional pairs}\label{alg}
To conclude this section, we describe an explicit algorithm (a sequence of steps) for computing the number of exceptional pairs. 
The following result is fundamental to our purpose, as it shows that any pair of modules satisfies exactly one of the two conditions considered in the previous subsections.

\begin{prop}\label{prop_homnotext}
Let \([a,b]\) and \([c,d]\) be the interval modules associated with two \(A\)-modules, and assume that \([a,b]\) is non-projective. Then:
\begin{enumerate}
  \item If \(\operatorname{Hom}([a,b],[c,d]) \cong \mathbb{F}\), then \(\operatorname{Ext}^i([a,b],[c,d]) \cong 0\) for all \(i \ge 1\).
  \item If \(\operatorname{Ext}^i([a,b],[c,d]) \cong \mathbb{F}\) for some \(i \ge 1\), then \(\operatorname{Hom}([a,b],[c,d]) \cong 0\).
\end{enumerate}
\end{prop}

\begin{proof}
Let \([a,b]\) and \([c,d]\) be the interval modules associated with two \(A\)-modules, and consider a projective resolution of \([a,b]\):
\[
0 \longrightarrow [a_m,b_m] \xrightarrow{f_m} [a_{m-1},b_{m-1}] \longrightarrow \cdots \longrightarrow
[a_1,b_1] \xrightarrow{f_1} [a_0,b_0] \xrightarrow{f_0} [a,b] \longrightarrow 0.
\]
First of all, assume that \(\operatorname{Hom}([a,b],[c,d]) \cong \mathbb{F}\). By Remark~\ref{remhom1} this is equivalent to
\(c \le a \le d \le b\).
Since \([a,b]\) is non-projective, its projective cover is \([a_0,b_0]\) with \(a_0=a\), and the next term in the resolution satisfies
\(a_1=b+1\).
Hence \(d < a_1\), and therefore \(d < a_i\) for every \(i \ge 1\).
It follows that \(\operatorname{Hom}([a_i,b_i],[c,d]) \cong 0\) for all \(i \ge 1\), which implies
\(\operatorname{Ext}^i([a,b],[c,d]) \cong 0\) for all \(i \ge 1\).

\noindent Conversely, suppose \(\operatorname{Ext}^i([a,b],[c,d]) \cong \mathbb{F}\) for some \(i \ge 1\).
By the description of \(\operatorname{Ext}^i\) via the projective resolution, this entails
\(\operatorname{Hom}([a_i,b_i],[c,d]) \cong \mathbb{F}\), i.e., \(c \le a_i \le d \le b_i\).
Since \(i \ge 1\), we have \(a_1=b+1 \le a_i \le d\), hence \(d \ge b+1\).
Therefore \(d > b\), and so \(\operatorname{Hom}([a,b],[c,d]) \cong 0\) by Remark~\ref{remhom1}.
\end{proof}

\noindent Using the results established in this section, we obtain the following algorithm to compute the number of exceptional pairs:

\begin{enumerate}
  \item Apply Theorem~\ref{thmind} to compute \(|\mathrm{Ind}(A)|\), the total number of indecomposable modules. Consequently, the total number of (ordered) pairs is \(|\mathrm{Ind}(A)|^2\).

  \item Apply Proposition~\ref{prophomab} to determine the number of pairs \(( M,N )\) such that \(\operatorname{Hom}(M,N) \cong \mathbb{F}\); denote this quantity by \(\mathcal{H}\).

  \item Use Proposition~\ref{cor3.9} to count the number of pairs \(( M,N )\) for which \(\operatorname{Ext}^i(M,N) \cong \mathbb{F}\) for some \(i \ge 1\); denote this quantity by \(\mathcal{E}\).

  \item By Propositions~\ref{prop_extji} and~\ref{prop_homnotext}, the sets counted in steps~2 and~3 are disjoint. Therefore, the total number of exceptional pairs equals
  \begin{equation} \label{key}
      |\mathrm{Ind}(A)|^2 \;-\; (\mathcal{H} + \mathcal{E}).
  \end{equation}
\end{enumerate}

\section{Some Representative Cases}\label{particularcases}

\subsection{Dynkin case of type $\mathbb{A}_n$ without relations}\label{sinrelaciones}
We now consider the concrete case of the standard (linear) orientation with no relations. 
This case was treated in the undergraduate thesis of O.~Bernal \cite{Bernal}, supervised by P. F. F. Espinosa, 
where the analysis proceeds by partitioning the set of indecomposable modules into subsets satisfying certain algebraic properties. 
In contrast, our approach is more combinatorial in spirit: we make essential use of the fact that the number of nonzero morphisms 
\(\operatorname{Hom}(M,M')\not\cong 0\) depends directly on the interval representing the corresponding module (see Remark~\ref{remhom1}).

\noindent Let \(M\) be an indecomposable \(\mathbb{A}_n\)-module represented by an interval \([a,b]\) with \(a \le b\).
As in Proposition~\ref{prophomab}, it is straightforward to verify that
\[
H_{[a,b]} = (b+1-a)(n+1-b).
\]
Consequently, the total number of pairs \((M,M')\) such that \(\operatorname{Hom}(M,M') \not\cong 0\) is
\begin{eqnarray*}
    \sum_{a=1}^n\sum_{b=a}^n(b+1-a)(n+1-b)&=&\sum_{a=1}^n\sum_{b=a}^n(b+1-a)(n-b)+\sum_{a=1}^n\sum_{b=a}^n(b+1-a)\\
    &=&\sum_{i=1}^n\sum_{j=1}^{n-i}ij+\frac{n(n+1)(n+2)}{6}\\
    &=&\frac{(n-1)n(n+1)(n+2)}{24}+\frac{n(n+1)(n+2)}{6}\\
    &=&\frac{n(n+1)(n+2)(n+3)}{24}= \binom{n+3}{4}.
\end{eqnarray*}

\noindent In the absence of relations, the projective resolution of an indecomposable module
\([a,b]\) over \(A\) is given by the short exact sequence of intervals
\[
0 \longrightarrow [b+1,n] \longrightarrow [a,n] \longrightarrow [a,b] \longrightarrow 0.
\]
By Proposition~\ref{cor3.9}, the total number of interval modules \([c,d]\) for which
\(\operatorname{Ext}^i([a,b],[c,d]) \not\cong 0\) for some \(i \ge 1\) equals
\[
(b+1)(n+1) - a(b+1) + 2(b+1)a - (b+1)^2 - a(n+1).
\]
A straightforward simplification shows that this expression equals
\[
(b+1-a)(n-b).
\]
Hence each \([a,b]\) contributes \((b+1-a)(n-b)\) ordered pairs. Summing over
\(1 \le a \le b \le n-1\), we obtain
\[
\sum_{a=1}^{n-1}\sum_{b=a}^{n-1} (b+1-a)(n-b)
\;=\; \frac{(n-1)n(n+1)(n+2)}{24}=\binom{n+2}{4}.
\]

\noindent Since the number of indecomposable modules in this case is $\frac{n(n+1)}{2}$, the algorithm of Subsection~\ref{alg} yields that the number of exceptional pairs equals

$$
\left( \frac{n(n+1)}{2} \right)^2-\binom{n+3}{4}-\binom{n+2}{4}=\frac{(n-1)n(n+1)^2}{6}.
$$

%\subsection{\huge{Resultados de un número fijo de relaciones}}

\subsection{Nakayama algebra of type I with ideal overlap-free}\label{nointer}
Let \(I\) be the ideal generated by \(k\) minimal valid relations such that, if
\(\alpha\) is an arrow appearing in the relation \(R_j\), then \(\alpha\) does not occur
in any other relation \(R_i\) for \(i\neq j\).
This condition is equivalent to requiring that, for every \(1\leq j\leq k\),
\[
  1 \leq p_j < i_{j+1}-i_j .
\]
We recall the partition of the indecomposable modules into the sets \(A_{P(i)}\),
for \(1\leq i\leq n\).
For each \(j\in\{1,\ldots,k\}\), we consider the subset
  $\displaystyle \bigcup_{i=i_{j-1}+1}^{\,i_j} A_{P(i)}$ .
A graphical representation of the corresponding portion of the Auslander–Reiten quiver
of \(A\) is as in Figure \ref{portion}.
\begin{figure}[h!]
    \begin{center}
\begin{tikzpicture}
[->,>=stealth',shorten >=1pt,auto,node distance=1.5cm and 2.2cm,thick,main node/.style=]
  \node[main node] (1) {$[i_j,i_j+p_j]$};   
 \node[main node] (2) [above right of=1] {$\ddots$};
  \node[main node] (3) [above right of=2] {$[i_{j-1}+1,i_j+p_j]$};
 \node[main node] (4) [below right of=1] {$\ddots$};
  \node[main node] (5) [below right of=4] {$[i_{j},i_{j}]$};
  \node[main node] (6) [above right of=5] {$\ddots$};
  \node[main node] (7) [below right of=3] {$\ddots$};
  \node[main node] (8) [below right of=7] {$[i_{j-1}+1,i_j]$};
  \node[main node] (9) [below right of=8] {$\ddots$};
  \node[main node] (10) [below right of=9] {$[i_{j-1}+1,i_{j-1}+1]$};

\path[every node/.style={font=\sffamily\small}]     
(1) edge node   {} (2)
 edge node   {} (4)
(2) edge node  {} (3)
(3) edge node   {} (7)
(4) edge node  {} (5)
(5) edge node   {} (6)
(6) edge node  {} (8)
(7) edge node  {} (8)
(8) edge node  {} (9)
(9) edge node  {} (10);
\end{tikzpicture}
\caption{A graphical representation of the Auslander–Reiten quiver of $A$ restricted to \( \bigcup_{i=i_{j-1}+1}^{\,i_j} A_{P(i)}\).}
\end{center}
\end{figure}
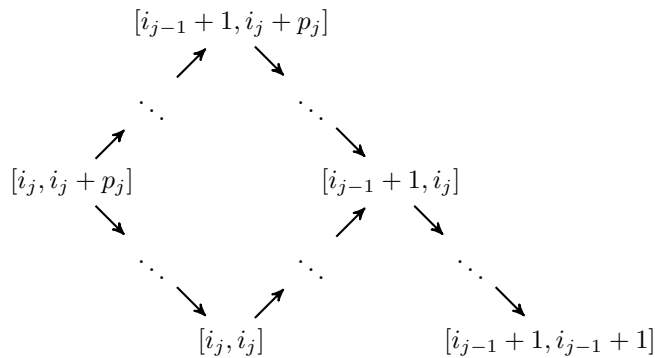\label{portion}

\noindent By Proposition~\ref{prophomab} we have
\[
H_{[a,b]} \;=\; (b+1-a)\,\bigl(i_j + p_j + 1 - b\bigr)
\quad\text{for}\quad
i_{j-1}+1 \le a \le b \le i_j,
\]
whereas, when \(a \le i_j < b \le i_{j+1}\),
\[
H_{[a,b]}
\;=\;
(i_j+p_j)\,(i_j+1-a)
\;+\;
(i_{j+1}+p_{j+1})\,(b-i_j)
\;+\;
(b+1-a)\,(1-b).
\]
For the first case $i_{j-1}+1 \le a \le b \le i_j$, we obtain:

\begin{align*}
\sum_{a=i_{j-1}+1}^{i_j}\sum_{b=a}^{i_j}(b+1-a)&(i_j+p_j+1-b)\\
&= \sum_{a=1}^{i_j-i_{j-1}}\sum_{b=a}^{i_j-i_{j-1}}
   (b+1-a)\bigl(i_j-i_{j-1}+p_j+1-b\bigr) \\[6pt]
&= \sum_{a=1}^{i_j-i_{j-1}}\sum_{b=a}^{i_j-i_{j-1}}
   (b+1-a)(i_j-i_{j-1}+1-b)
   + p_j\!\!\sum_{a=1}^{i_j-i_{j-1}}\sum_{b=a}^{i_j-i_{j-1}}\!(b+1-a) \\[6pt]
&= \frac{(i_j-i_{j-1})(i_j-i_{j-1}+1)(i_j-i_{j-1}+2)(i_j-i_{j-1}+3)}{24} \\
&\quad + p_j\,\frac{(i_j-i_{j-1})(i_j-i_{j-1}+1)(i_j-i_{j-1}+2)}{6} \\[6pt]
&= \binom{i_j - i_{j-1} + 3}{4} + p_j \binom{i_j - i_{j-1} + 2}{3}.
\end{align*}

%___________ Versión anterior ecuaciones.
\begin{comment}
\begin{eqnarray*}
    &\sum_{a=i_{j-1}+1}^{i_j}\sum_{b=a}^{i_j}(b+1-a)(i_j+p_j+1-b)&\\&=&\\&\sum_{a=1}^{i_j-i_{j-1}}\sum_{b=a}^{i_j-i_{j-1}}(b+1-a)(i_j-i_{j-1}+p_j+1-b)&\\
    &=&\\
    &\sum_{a=1}^{i_j-i_{j-1}}\sum_{b=a}^{i_j-i_{j-1}}(b+1-a)(i_j-i_{j-1}+1-b)+p_j\sum_{a=1}^{i_j-i_{j-1}}\sum_{b=a}^{i_j-i_{j-1}}(b+1-a)&\\
    &=&\\
    &\frac{(i_j-i_{j-1})(i_j-i_{j-1}+1)(i_j-i_{j-1}+2)(i_j-i_{j-1}+3)}{24}+p_j\frac{(i_j-i_{j-1})(i_j-i_{j-1}+1)(i_j-i_{j-1}+2)}{6}&\\
    &=&\\
    &\binom{i_j-i_{j-1}+3}{4}+p_j\binom{i_j-i_{j-1}+2}{3}&
\end{eqnarray*}
%______________________
\end{comment}

\noindent For the second case, \(a \le i_j < b \le i_{j+1}\), we can split the expression into three terms:
\begin{eqnarray*}
\sum_{a=i_{j-1}+1}^{i_j}\sum_{b=i_j+1}^{i_j+p_j}(i_j+p_j)(i_j+1-a)&=&(i_j+p_j)p_jt_{i_j-i_{j-1}},\\
\sum_{a=i_{j-1}+1}^{i_j}\sum_{b=i_j+1}^{i_j+p_j}(i_{j+1}+p_{j+1})(b-i_j)&=&(i_{j+1}+p_{j+1})(i_j-i_{j-1})t_{p_j}
\end{eqnarray*}
where $t_n$ means the $n$-th triangular number and
\begin{align*}
\sum_{a=i_{j-1}+1}^{i_j}\sum_{b=i_j+1}^{i_j+p_j}(b-a+1)(1-b)
&= \sum_{a=1}^{i_j-i_{j-1}}\sum_{b=1}^{p_j}(b+1+(i_j-i_{j-1})-a)(1-b-i_j) \\[6pt]
&= \sum_{a=1}^{i_j-i_{j-1}}\sum_{b=1}^{p_j}\bigl(-b^2 + b(a+i_{j-1}-2i_j) + (a-i_j+i_{j-1}-1)(i_j-1)\bigr) \\[6pt]
&= \sum_{a=1}^{i_j-i_{j-1}}\Bigl[\frac{2p_j+1}{3}(-t_{p_j}) + t_{p_j}(a+i_{j-1}-2i_j) \\
&\qquad\qquad\qquad\qquad\qquad\qquad +\; p_j(a-i_j+i_{j-1}-1)(i_j-1)\Bigr] \\[6pt]
&= \sum_{a=1}^{i_j-i_{j-1}}\bigl[t_{p_j}+p_j(i_j-1)\bigr]a
   + t_{p_j}\Bigl(i_{j-1}-2i_j-\frac{2p_j+1}{3}\Bigr)\\ 
   & \qquad\qquad\qquad\qquad\qquad\qquad - p_j(i_j-i_{j-1}+1)(i_j-1) \\[6pt]
&= [t_{p_j}-p_j(i_j-1)]\,t_{i_j-i_{j-1}}
   + (i_j-i_{j-1})\,t_{p_j}\Bigl(i_{j-1}-2i_j-\frac{2p_j+1}{3}\Bigr),
\end{align*}

%__________________
\begin{comment}
\begin{eqnarray*}
&\displaystyle \sum_{a=i_{j-1}+1}^{i_j}\sum_{b=i_j+1}^{i_j+p_j}(b-a+1)(1-b)&\\
&=&\\
&\sum_{a=1}^{i_j-i_{j-1}}\sum_{b=1}^{p_j}(b+1+(i_j-i_{j-1})-a)(1-b-i_j)&\\
&=&\\
&\sum_{a=1}^{i_j-i_{j-1}}\sum_{b=1}^{p_j}(-b^2)+b(a+i_{j-1}-2i_j)+(a-i_j+i_{j-1}-1)(i_j-1)&\\
&=&\\
&\sum_{a=1}^{i_j-i_{j-1}}\frac{2p_j+1}{3}(-t_{p_j})+(t_{p_j})(a+i_{j-1}-2i_j)+p_j(a-i_j+i_{j-1}-1)(i_j-1)&\\
&=&\\
&\sum_{a=1}^{i_j-i_{j-1}}[t_{p_j}+p_j(i_j-1)](a)+t_{p_j}\left(i_{j-1}-2i_j-\frac{2p_j+1}{3}\right)-p_j(i_j-i_{j-1}+1)(i_j-1)&\\
&=&\\
&[t_{p_j}-p_j(i_j-1)](t_{i_j-i_{j-1}})+(i_j-i_{j-1})t_{p_j}\left(i_{j-1}-2i_j-\frac{2p_j+1}{3}\right)&
%&(i_j-i_{j-1})p_j\left((p_j+1)\frac{3i_{j-1}-9i_j+1-4p_j}{12}-\frac{(i_j-1)(i_j-i_{j-1}+1)}{2}\right)&
\end{eqnarray*}
    
\end{comment}
%____________________

\noindent Let \(H_j\) denote the number of pairs of modules with a nonzero homomorphism
contributed by \(\displaystyle \bigcup_{i=i_{j-1}+1}^{\,i_j} A_{P(i)}\), for \(1\le j\le k\).
Then
\[
\begin{aligned}
H_j
&=
\binom{i_j - i_{j-1} + 3}{4}
\;+\;
p_j \binom{i_j - i_{j-1} + 2}{3}
\;+\;
3\,t_{p_j}\, t_{\,i_j - i_{j-1}}
\\[2mm]
&\quad
+\,
\bigl(i_j - i_{j-1}\bigr)\, t_{p_j}\!
\left(
  i_{j+1} + p_{j+1} + i_{j-1} - 2 i_j - \frac{2p_j + 1}{3}
\right).
\end{aligned}
\]
% (If needed: where \(t_m\) denotes the m-th triangular number, \(t_m=\binom{m+1}{2}\).)

\noindent Finally, let \(H_{k+1}\) denote the number of pairs of modules with a nonzero
homomorphism contributed by
\(\displaystyle \bigcup_{i=i_k+1}^{\,n} A_{P(i)}\).
Then \(H_{k+1}\) is given by
\begin{eqnarray*}
    \sum_{a=i_k+1}^n\sum_{b=a}^n(b+1-a)(n+1-b)&=&\sum_{a=1}^{n-i_k}\sum_{b=a}^{n-i_k}(b+1-a)(n-i_k+1-b)\\
    &=&\binom{n-i_k+3}{4}.
\end{eqnarray*}
From the preceding computations, and recalling that we write \(t_m\) for the \(m\)-th
triangular number and adopt the conventions \(i_{k+1}=n\) and \(i_0=p_{k+1}=0\),
we obtain the following Theorem, which provides an explicit closed formula for \(\mathcal{H}\).

\begin{thm}\label{thm3.7}
Let \(A=\mathbb{F}Q/I\) be a path algebra and \(I\) the ideal generated by \(k\) minimal valid relations which are {overlap–free}. For each \(1\le j\le k\) let \(R_j=\alpha_{i_j}\alpha_{i_j+1}\cdots\alpha_{i_j+p_j}\). 
Then the number of pairs of modules \((M,N)\) such that \(\operatorname{Hom}_A(M,N)\cong\mathbb{F}\) is
\[
\mathcal{H}
\;=\;
\sum_{j=1}^{k+1}
\biggl[
\binom{i_j - i_{j-1} + 3}{4}
\;+\;
p_j\,\binom{i_j - i_{j-1} + 2}{3}
\;+\;
t_{p_j}\!\left(3\,t_{\,i_j - i_{j-1}} + (i_j - i_{j-1})\,g(j)\right)
\biggr],
\]
where the auxiliary function \(g(j)\) is defined by
\[
g(j) \;=\; i_{j+1} + p_{j+1} + i_{j-1} - 2\,i_j - \frac{2p_j+1}{3},
\qquad 1\le j \le k,
\]
and, by convention, \(g(k+1)=0\).
\end{thm}

\noindent Now we would like to calculate $\mathcal{E}$. To do this, we need to see what form the projective resolution of the module has. Let $[a,b]$ be the interval representing an indecomposable module. We have that if $i_{j-1}+1\leq a\leq b\leq i_j-1$, then its projective resolution is of the form:
$$0\rightarrow[b+1,i_j+p_j]\rightarrow[a,i_j+p_j]\rightarrow[a,b]\rightarrow0,$$
on the other hand $i_{j-1}+1\leq a\leq i_j\leq b\leq i_j+p_j-1$
$$0\rightarrow[i_j+p_j+1,i_{j+1}+p_{j+1}]\rightarrow[b+1,i_{j+1}+p_{j+1}]\rightarrow[a,i_j+p_j]\rightarrow[a,b]\rightarrow0.$$
Note that if $b=i_j+p_j$, for $P$ the module corresponding to $[a,b]$ is projective, so  $\operatorname{Ext}^i(P,M)\cong0$ for all modules  $M$. Since the sets \(A_{P(i)}\) form a partition of the indecomposable \(A\)-modules, 
we compute the number of pairs contributed by the blocks
\(\displaystyle \bigcup_{i=i_{j-1}+1}^{\,i_j} A_{P(i)}\) for \(1\le j\le k+1\).
This contribution is established in the next two lemmas.

\begin{lem}\label{lemma3.8}
Let \(A=\mathbb{F}Q/I\) be a path algebra and let \(I\) be the ideal generated by \(k\) minimal valid relations which are \emph{overlap–free}. For each \(1\le j\le k\) let \(R_j=\alpha_{i_j}\alpha_{i_j+1}\cdots\alpha_{i_j+p_j}\), and  the number \(\mathcal{E}_j\) of pairs of modules \((M,M')\) such that
\(\operatorname{Ext}^l_A(M,M') \cong\mathbb{F}\) for some \(l \ge 1\), with
\(M \in \bigcup_{i=i_{j-1}+1}^{\,i_j} A_{P(i)}\), is given by
\[
\mathcal{E}_j
=
\binom{i_j - i_{j-1} + 2}{4}
\;+\;
p_j \binom{i_j - i_{j-1} + 1}{3}
\;+\;
t_{p_j}\!\Big( (i_j-i_{j-1})\,(g(j)-1) + 3\,t_{\,i_j-i_{j-1}} \Big),
\]
where
\[
g(j)= i_{j+1}+p_{j+1}+i_{j-1}-2i_j-\frac{2p_j+1}{3}.
\]
\end{lem}
\begin{proof}    
Let \(M\) be an indecomposable module represented by the interval \([a,b]\).
Assume that \(i_{j-1}+1 \le a \le b \le i_j-1\).
It follows from Proposition~\ref{cor3.9} that the total number of modules \([c,d]\)
for which \(\mathrm{Ext}^i([a,b],[c,d]) \cong\mathbb{F}\) for some \(i \ge 1\) is given by
\[
(b+1)(i_j+p_j+1) - a(b+1) + 2(b+1)a - (b+1)^2 - a(i_j+p_j+1).
\]
Hence, the module represented by \([a,b]\) contributes
\[
(i_j + p_j - b)\,(b+1 - a).
\]
Since \(i_{j-1}+1 \le a \le b \le i_j-1\), the total contribution of all such intervals \([a,b]\)
is
\[
\sum_{a=i_{j-1}+1}^{\,i_j-1}\;\sum_{b=a}^{\,i_j-1} (i_j + p_j - b)\,(b+1 - a),
\]
which, after simplification, can be expressed in closed form as
\begin{eqnarray*}
    \sum_{a=i_{j-1}+1}^{i_j-1}\sum_{b=a}^{i_j-1}(i_j+p_j-b)(b-a+1)&=&\binom{i_j - i_{j-1} + 2}{4}+ p_j \binom{i_j - i_{j-1} + 1}{3}
\end{eqnarray*}
On the other hand, for the case 
\(i_{j-1}+1 \le a \le i_j \le b \le i_j+p_j-1\),
the projective resolution of the module \([a,b]\) is
\[
0 \longrightarrow [\,i_j+p_j+1,\, i_{j+1}+p_{j+1}\,]
\longrightarrow [\,b+1,\, i_{j+1}+p_{j+1}\,]
\longrightarrow [\,a,\, i_j+p_j\,]
\longrightarrow [\,a,b\,]
\longrightarrow 0.
\]
As in the previous case, it follows from Proposition~\ref{cor3.9} that
the number of modules \([c,d]\) such that
\(\mathrm{Ext}^i([a,b],[c,d]) \cong\mathbb{F}\) for some \(i \ge 1\) is
\[
\begin{aligned}
&(i_j+p_j+1)(i_{j+1}+p_{j+1}+1)
- a(b+1)
+ 2\bigl((b+1)a + (b+1)(i_j+p_j+1)\bigr) \\[1ex]
&\quad
- (b+1)^2
- (i_j+p_j+1)^2
- a(i_j+p_j+1)
- (b+1)(i_{j+1}+p_{j+1}+1).
\end{aligned}
\]
\noindent After straightforward simplifications, this expression reduces to
\[
(i_j+p_j-b)\,\bigl(b+1 - a + i_{j+1}+p_{j+1} - (i_j+p_j)\bigr).
\]

\noindent Since \(i_{j-1}+1 \le a \le i_j \le b \le i_j+p_j-1\), the contribution of this case 

\begin{align*}
&\sum_{a=i_{j-1}+1}^{i_j}\sum_{b=i_j}^{i_j+p_j-1}(i_j+p_j-b)(b+1-a+i_{j+1}+p_{j+1}-(i_j+p_j))\\
&= \sum_{a=1}^{i_j-i_{j-1}}\sum_{b=1}^{p_j}(p_j+1-b)(b-a+i_{j+1}+p_{j+1}-(i_{j-1}+p_j)) \\[6pt]
&= \sum_{a=1}^{i_j-i_{j-1}}\sum_{b=1}^{p_j}
   \Bigl[-b^2 + b\bigl(2p_j+i_{j-1}+1+a-(i_{j+1}+p_{j+1})\bigr) \\
&\qquad\qquad\qquad + (p_j+1)\bigl(i_{j+1}+p_{j+1}-(i_{j-1}+p_j+a)\bigr)\Bigr] \\[6pt]
&= t_{p_j}\sum_{a=1}^{i_j-i_{j-1}}
   \Bigl[i_{j+1}+p_{j+1}+1-\Bigl(\tfrac{2p_j+1}{3}+i_{j-1}+a\Bigr)\Bigr] \\[6pt]
&= t_{p_j}(i_j-i_{j-1})
   \Bigl[i_{j+1}+p_{j+1}+1-\Bigl(\tfrac{2p_j+1}{3}+i_{j-1}\Bigr)\Bigr]
   - t_{p_j}t_{\,i_j-i_{j-1}} \\[6pt]
&= t_{p_j}\!\Bigl((i_j-i_{j-1})\,(g(j)-1) + 3\,t_{\,i_j-i_{j-1}}\Bigr),
\end{align*}

%%%%%%%%%%%%%%%%%%%5
\begin{comment}
    \begin{eqnarray*}
    &\sum_{a=i_{j-1}+1}^{i_j}\sum_{b=i_j}^{i_j+p_j-1}(i_j+p_j-b)(b+1-a+i_{j+1}+p_{j+1}-(i_j+p_j))&\\
    &=&\\
    &\sum_{a=1}^{i_j-i_{j-1}}\sum_{b=1}^{p_j}(p_j+1-b)(b-a+i_{j+1}+p_{j+1}-(i_{j-1}+p_j))&\\
    &=&\\
    &\sum_{a=1}^{i_j-i_{j-1}}\sum_{b=1}^{p_j}[(-b^2)+b[2p_j+i_{j-1}+1+a-(i_{j+1}
+p_{j+1})]&\\
&+&\\
&(p_j+1)(i_{j+1}+p_{j+1}-(i_{j-1}+p_j+a))]&\\
 &=&\\
    &t_{p_j}\sum_{a=1}^{i_j-i_{j-1}}[i_{j+1}+p_{j+1}+1-\left(\frac{2p_j+1}{3}+i_{j-1}+a\right)]&\\
&=&\\
&t_{p_j}(i_j-i_{j-1})[i_{j+1}+p_{j+1}+1-\left(\frac{2p_j+1}{3}+i_{j-1}\right)]-t_{p_j}t_{i_j-i_{j-1}}&\\
&=&\\
&t_{p_j}\!\Big( (i_j-i_{j-1})\,(g(j)-1) + 3\,t_{\,i_j-i_{j-1}} \Big).&
\end{eqnarray*}

%%%%%%%%%%%%%%%%%%%%%5
\end{comment}

\end{proof}

\noindent On the other hand, it is straightforward to see that if 
\(M' \in \bigcup_{i=i_{j-1}+1}^{\,i_j} A_{P(i)}\) and 
\(M \in \bigcup_{i=i_k+1}^{\,n} A_{P(i)}\),
then \(\operatorname{Ext}^l_A(M, M') \cong 0\) for all \(l \ge 1\).
Hence, in order for \(\operatorname{Ext}^l_A(M, M') \cong \mathbb{F}\) for some \(l \ge 1\),
both \(M\) and \(M'\) must lie in 
\(\bigcup_{i=i_k+1}^{\,n} A_{P(i)}\).
Therefore, the total number of pairs contributed by this terminal block is
\[
\sum_{a=i_k+1}^{\,n-1}\;\sum_{b=a}^{\,n-1} (b+1-a)\,(n-b)
\;=\;
\sum_{a=1}^{\,n-i_k-1}\;\sum_{b=a}^{\,n-i_k-1} (b+1-a)\,(n-i_k-b)
\;=\;
\binom{n-i_k+2}{4}.
\]

\begin{lem}\label{lemma3.9}
Let \(A=\mathbb{F}Q/I\) as above.
The total contribution of pairs \((M,M')\) with \(\operatorname{Ext}^l_A(M,M')\cong\mathbb{F}\) for some \(l\ge 1\),
coming from the set
\(\displaystyle \bigcup_{i=i_k+1}^{\,n} A_{P(i)}\),
is
    $$\mathcal{E}_{k+1}=\binom{n-i_k+2}{4}.$$
\end{lem}

% --- English version (journal style, optional) ---
\noindent Using \(p_{k+1}=0\) and Lemmas~\ref{lemma3.8} and~\ref{lemma3.9}, we conclude:

\begin{thm}\label{thm3.10}
Let \(A=\mathbb{F}Q/I\) be a path algebra and \(I\) the ideal generated by \(k\) minimal valid relations which are \emph{overlap–free}. For each \(1\le j\le k\) let \(R_j=\alpha_{i_j}\alpha_{i_j+1}\cdots\alpha_{i_j+p_j}\).

\noindent Then the number of pairs of indecomposable modules \((M,M')\) such that
\(\operatorname{Ext}^l(M,M')\cong \mathbb{F}\) for some \(l\ge 1\), is
\[
\mathcal{E}
=
\sum_{j=1}^{k+1}
\biggl[
\binom{i_j - i_{j-1} + 2}{4}
+
p_j \binom{i_j - i_{j-1} + 1}{3}
+
t_{p_j}\!\bigl((i_j-i_{j-1})(g(j)-1)+3\,t_{\,i_j-i_{j-1}}\bigr)
\biggr],
\]
where
\[
g(j)
=
i_{j+1}+p_{j+1}+i_{j-1}-2\,i_j-\frac{2p_j+1}{3},
\qquad 1\le j\le k,
\]
and, by convention, \(g(k+1)=0\).
\end{thm}

\noindent The following theorem consolidates the results of this section and provides
an explicit procedure to compute the number of exceptional pairs.
\begin{thm}\label{thmnointer}
Let \(A=\mathbb{F}Q/I\) be a path algebra and let \(I\) be the ideal generated by \(k\)
minimal valid relations which are \emph{overlap–free}. For each \(1\le j\le k\) let \(R_j=\alpha_{i_j}\alpha_{i_j+1}\cdots\alpha_{i_j+p_j}\). Then the number of exceptional pairs is

\begin{align*}
&\Biggl[\Biggl(\sum_{j=1}^{k+1} (i_j - i_{j-1})(i_j + p_j)\Biggr) - t_{n-1}\Biggr]^2 \\[4pt]
&\quad - \sum_{j=1}^{k+1}
   \Biggl[\binom{i_j - i_{j-1} + 3}{4}
   + \binom{i_j - i_{j-1} + 2}{4}
   + p_j\!\left(\binom{i_j - i_{j-1} + 2}{3}
   + \binom{i_j - i_{j-1} + 1}{3}\right)\Biggr] \\[4pt]
&\quad - \sum_{j=1}^{k+1}
   \Bigl[6\,t_{\,i_j - i_{j-1}} + (i_j - i_{j-1})(2g(j) - 1)\Bigr] t_{p_j}
\end{align*}

\begin{comment}
    \begin{eqnarray*}
    &\left[\left(\sum_{j=1}^{k+1} (i_j - i_{j-1})(i_j + p_j)\right)-t_{n-1}\right]^2\\
    &-\sum_{j=1}^{k+1}\binom{i_j-i_{j-1}+3}{4}+\binom{i_j-i_{j-1}+2}{4}+p_j\left[\binom{i_j-i_{j-1}+2}{3}+\binom{i_j-i_{j-1}+1}{3}\right]\\
    &-\sum_{j=1}^{k+1}[6t_{i_j-i_{j-1}}+(i_j-i_{j-1})(2g(j)-1)]t_{p_j}
\end{eqnarray*}
\end{comment}
\noindent where the auxiliary function \(g(\,\cdot\,)\) is given by
\[
g(j)=i_{j+1}+p_{j+1}+i_{j-1}-2\,i_j-\frac{2p_j+1}{3},
\qquad 1\le j\le k,
\]
and, by convention  \(g(k+1)=0\) and \(i_0=0\), \(i_{k+1}=n\), \(p_{k+1}=0\).
\end{thm}
\begin{proof}
    Its follow by algorithm of section \ref{alg} and Theorems \ref{thmind}, \ref{thm3.7} and \ref{thm3.10}.
\end{proof}

\noindent When in Theorem \ref{thmnointer} we assume that there is a single relation  starting at vertex $1$ and having length $m=1+p$, hereafter denoted by $A_{R_{(1,m)}}$, we have the following result:

\begin{cor}\label{cor1}
The number of exceptional pairs of an algebra of type $A_{R_{(1,m)}}$ is: 
\[\frac{1}{6}\left(4m^3-3m^2(2 n+1)+m\left(6 n^2-1\right)+(n-2)(n-1) n^2\right)\]
\end{cor}

\begin{proof}
According to Theorem \ref{thmind}, the number of indecomposable modules $\mathrm{Ind}(A_{R_{(1,m)}})$ is given by $\left(m+\frac{n(n-1)}{2}\right)$ and therefore the total number of possible pairs of indecomposable modules is $\left(m+\frac{n(n-1)}{2}\right)^2$. \par 

\noindent Similarly, according to Theorem \ref{thm3.7}, the number $\mathcal{H}$ of pairs $(M,N)$ such that $\mathrm{Hom}(M,N)\cong\mathbb{F}$ is 
$$m(m-n)+\frac{m(m+1)(3n+2-2m)}{6}+\frac{(n-1)(n)(n+1)(n+2)}{24}$$
and according to Theorem \ref{thm3.10}, the number $\mathcal{E}$ of pairs $(M,N)$ such that $\mathrm{Ext}(M,N)\cong\mathbb{F}$ is

$$m(m-1)\left(\frac{3n-2m+1}{6}\right)+\frac{(n-2)(n-1)(n)(n+1)}{24}$$

\noindent Hence, the number of exceptional pairs of $A_{R_{(1,m)}}$ is given by: 

\begin{multline*}
|\mathrm{Ind}(A)|^2-(\mathcal{H}+\mathcal{E})= \left(m+\frac{n(n-1)}{2}\right)^2-m(m-1)\left(\frac{3n-2m+1}{6}\right)-\frac{(n-2)(n-1)(n)(n+1)}{24}\\ 
-\left(m(m-n)+\frac{m(m+1)(3n+2-2m)}{6}+\frac{(n-1)(n)(n+1)(n+2)}{24}\right)
\end{multline*}
which becomes: 

\begin{multline*}
 = \left(m+\frac{n(n-1)}{2}\right)^2-m(m-1)\left(\frac{3n-2m+1}{6}\right)-\frac{(n-2)(n-1)(n)(n+1)}{24}\\ 
-m(m-n)-\frac{m(m+1)(3n+2-2m)}{6}-\frac{(n-1)(n)(n+1)(n+2)}{24}
\end{multline*}
or equivalently

\begin{multline*}
 =\frac{6m^2+12mn^2-6mn+3n^4-6n^3+3n^2-6m^2n+4m^3+2m}{12}+\frac{n^2-n^4}{12}\\ +\frac{-6m^2+2m^3+3nm-3nm^2-2m}{6}
\end{multline*}
\begin{multline*}
=\frac{3m^2+6mn^2-3mn+n^4-3n^3+2n^2-3m^2n+2m^3+m}{6}\\+\frac{-6m^2+2m^3+3nm-3nm^2-2m}{6}
\end{multline*}

\noindent which simplifies to 
 \[|\mathrm{Ind}(A)|^2-(\mathcal{H}+\mathcal{E})=\frac{4m^3-3m^2(2n+1)+m(6n^2-1)+n^2(n-2)(n-1)}{6}\]

\end{proof}

\normalcolor
    
\subsection{Auslander Algebras}\label{Auslanderalg}

\noindent Recall that according to \cite{Chen}  $A$ is the Auslander algebra of a radical square zero algebra of type $\mathbb{A}_n$.
Then $A$ is given by the following quiver $Q$ :
$$
1 \stackrel{a_1}{\rightarrow} 2 \stackrel{a_2}{\rightarrow} \cdots \stackrel{a_{2 m-3}}{\rightarrow} 2 m-2 \stackrel{a_{2 m-2}}{\rightarrow} 2 m-1
$$
\noindent with the relations:
$$
a_{2 j-1} a_{2 j}=0 \quad (1 \leq j \leq m-1)
$$

\noindent  In this way,  a few straightforward calculations show that 
\( p_j = 1 \) for \( 1 \leq j \leq m-1 \), \( p_m = 0 \), 
\( i_1 - i_0 = 1 \), and \( i_j - i_{j-1} = 2 \) for \( 2 \leq j \leq m \). 
Moreover, the auxiliary function is determined by
\[
g(j) =
\begin{cases}
0, & 2 \leq j \leq m-2, \\[4pt]
1, & j = 1, \\[4pt]
-1, & j = m-1.
\end{cases}
\]
Therefore, from Theorem~\ref{thmind}, \ref{thm3.7} and  \ref{thm3.10}, we obtain that the number of indecomposable modules, $\mathcal{H}$ and $\mathcal{E}$ are respectively
\begin{eqnarray*}
    |\mathrm{Ind}A|&=&\sum_{l=1}^{m} (i_l - i_{l-1})(i_l + p_l) 
- \frac{(2m-1)(2m-2)}{2}
= 5(m-1),\\
\mathcal{H}&=&\left\{\begin{array}{lc}
    10 &  m=2\\
    18m-27 & m\geq 3
\end{array}\right.\\
\mathcal{E}&=&\left\{\begin{array}{lc}
    3 &  m=2\\
    9m-16 & m\geq 3
\end{array}\right.
\end{eqnarray*}
Consequently, the number of exceptional pairs is $12$ for $m=2$ and for $m\geq 3$:
\[
25(m-1)^2-(18m-27)-(9m-16)\;=\;25m^2-77m+68.
\]

\section{About radical square zero Nakayama algebra}\label{rad2}

In this section, we follow the algorithm presented in the end of Section \ref{General} to find out the number expectional pairs for any radical square zero Nakayama algebra. To do that, we enumerate the number of indecomposable modules and we find the values of $\mathcal{H}$ and $\mathcal{E}$.

\begin{thm}\label{Indrad}
The number of indecomposable modules over a Nakayama Algebra with $n$ vertices and a relation of type  $\mathrm{Rad}^k$ with $k<n$ is $k\cdot n-t_{k-1}$
\end{thm}

\begin{proof}
For this case, we have $n-k$ relations of length $k$ of the form $\alpha_i\alpha_{i+1}\cdots\alpha_{i+k-1}$. From Theorem \ref{thmind}, it follows that:
\begin{eqnarray*}
    |\mathrm{Ind} (A)|&=&\sum_{l=1}^{n-k+1}(i_l-i_{l-1})(i_l+p_l)-t_{n-1}\\
    &=&\left(\sum_{l=1}^{n-k}(1)(k-1+l)\right)+(n-(n-k))(n)-t_{n-1}\\
    &=&\left(\sum_{l=k}^{n-1}l\right)+kn-t_{n-1}\\
    &=&t_{n-1}-t_{k-1}+kn-t_{n-1}\\
    &=& kn-t_{k-1}.
\end{eqnarray*}
\end{proof}

\noindent Now, for the values of $\mathcal{H}$ and $\mathcal{E}$ we got the following result: 

\begin{thm}\label{Hrad2}
Let $A = \mathbb{F}Q/I$ be an algebra, where $Q$ is a Dynkin quiver of type $\mathbb{A}_n$ with a linear orientation, and $I$ is an admissible ideal generated by all possible relations of length two then $\mathcal{H}= 5n-5$ and $\mathcal{E}= (n-1)+(n-2)+t_{n-3}$
\end{thm}
\begin{proof}
We consider that a radical squared zero Nakayama algebra with $n$ vertices has the following Auslander-Reiten quiver: 
\par\bigskip
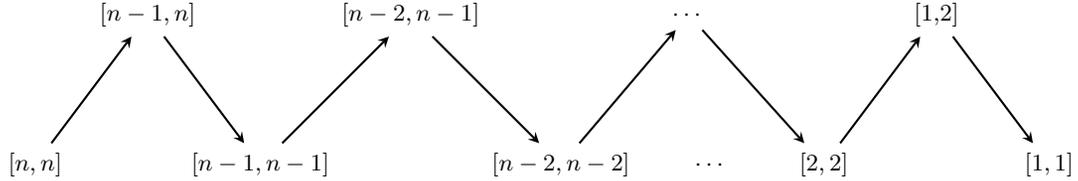
\begin{tikzpicture}[
  node distance=1cm and 1cm,
  every node/.style={font=\small},
  every path/.style={->, >=stealth, thick}
  ]

% Piso 1
\node (n) at (-7,4) {$[n,n]$};
\node (n-1) at (-4,4) {$[n-1,n-1]$};
\node (n-2) at (0,4) {$[n-2,n-2]$};
\node (4) at (2,4) {$\cdots$};
\node (2) at (3.5,4) {$[2,2]$};
\node (1) at (6.5,4) {$[1,1]$};

%Piso 2 

\node (n-1n) at (-5.5,6)  {$[n-1,n]$};
\node (n-2n-1) at (-2,6) {$[n-2,n-1]$};
\node (5) at (1.7,6) {$\cdots$};
\node (12) at (5,6) {[1,2]};
%Flechas
\draw (2) -- (12);
\draw (12) -- (1);
\draw (n-2) -- (5);
\draw (5) -- (2);
\draw (n) -- (n-1n);
\draw (n-1n) -- (n-1);
\draw (n-1) -- (n-2n-1);
\draw (n-2n-1) -- (n-2);
\end{tikzpicture}

\noindent Note that we have $2n-1$ indecomposable modules of which $n-1$ are projective-injective ($[i,i+1]$, for $1\leq i\leq n-1$), $1$ is pure projective($[n,n]$), $1$ pure injective ($[1,1]$) and $n-2$ simples. 

\noindent  We divide the proof in two parts. In the first one we calculate $\mathcal{H}$ ans in the second we compute $\mathcal{E}$. 
\begin{itemize}
    \item For the number $\mathcal{H}$ it is enough to consider the following cases for the expression $\mathrm{Hom}(L,M)$:

    \begin{enumerate}
        \item When $L=M$. In this case, since there are $2n-1$   indecomposable module so this case contributes with $2n-1$ to the value of $\mathcal{H}$.
        
        \item When $L$ and $M$ are two different projective modules. According to Remark \ref{remhom1} $\mathrm{Hom}(P(i),P(j))$ is different to zero if $j=i-1$ and it is possible for $2 \leq i \leq n$ which means that this case contributes with $n-1$ to the value of  $\mathcal{H}$. 
        
        \item When $L$ is projective and $M$ is a simple module. In concordance with Remark \ref{remhom1} $\mathrm{Hom}(P(i),S(j))$ is different to zero if and only if $j=i$ and it is possible for $1\leq i \leq n-1$ which means that this case contributes with $n-1$ to the value of  $\mathcal{H}$. 
        
        \item  When $L$ is a simple module and $M$ is a projective module. According to Remark \ref{remhom1} $\mathrm{Hom}(S(i),P(j))$ is different to zero if $i=j+1$ and it is possible for $2 \leq i \leq n-1$ which means that this case contributes with $n-2$ to the value of  $\mathcal{H}$.
    \end{enumerate} 

    \noindent Adding all the cases described above $\mathcal{H}=(2n-1)+(n-1)+(n-1)+(n-2)=5n-5$.

    \par\bigskip

    \item For the number $\mathcal{E}$ considering the basic properties of $\mathrm{Ext}$ functor for projective and injective modules it is enough consider the following cases for the expression $\mathrm{Ext}(L,M)$: \par\bigskip

    \begin{enumerate}
        \item $\mathrm{Ext}(S(i),P(n))$ where $S(i)$ denotes the simple module associated to the vertex $i$ and $P(n)$ denotes the projective module associated to $n$. In this case, we consider the projective resolution of $S(i)$
\[
0 \to P(n) \to P(n-1) \to P(n-2) \to \cdots \to P(i+1) \to P(i) \to S(i) \to 0
\]

Applying the functor $ \textrm{Hom}_{A}(-,P(n))$,  we have the exact sequence
\[
0 \to \textrm{Hom}_{A}(P(i),P(n)) \to  \textrm{Hom}_{A}(P(i+1),P(n)) \to \cdots \to  \textrm{Hom}_{A}(P(n),P(n)) \to 0.
\]

Note that, $\mathrm{Hom}(P(j),P(n))=0$ when $j<n$. Thus, $\textrm{Ext}^{1}_{A}(S(i),P(n)) \cong \mathbb{F}$ for a fixed $i$. Thus, this case contributes with $n-1$ to the value of $\mathcal{E}$.

 \item $\mathrm{Ext}(S(1),S(i))$ where $S(i)$ denotes the simple module associated to the vertex $i$ with $2 \leq i \leq n-1$. In this case, we consider the projective resolution of $S(1)$
\[
0 \to P(n) \to P(n-1) \to P(n-2) \to \cdots \to P(2) \to P(1) \to S(1) \to 0
\]

Applying the functor $ \textrm{Hom}_{A}(-,S(i))$,  we have the exact sequence
\[
0 \to \textrm{Hom}_{A}(P(1),S(i)) \to  \textrm{Hom}_{A}(P(2),S(i)) \to \cdots \to \textrm{Hom}_{A}(P(n-1),S(i)) \to \textrm{Hom}_{A}(P(n),S(i)) \to 0.
\]

Note that, $\mathrm{Hom}(P(j),S(i))\cong\mathbb{F}$ when $j=i$. Thus, $\textrm{Ext}^{1}_{A}(S(1),S(i)) \cong \mathbb{F}$ for a fixed $i$. Thus, this case contributes with $n-2$ to the value of $\mathcal{E}$.

 \item $\mathrm{Ext}(S(i),S(j))$ where $S(i)$ and $S(j)$ denote the simple modules associated to the vertex $i$ and $j$ respectively with $2 \leq i,j \leq n-1$.

In this case, we consider the projective resolution of $S(i)$
\[
0 \to P(n) \to \cdots \to P(i+1) \to P(i) \to S(i) \to 0
\]

Applying the functor $ \textrm{Hom}_{A}(-,S(j))$,  we have the exact sequence
\[
0 \to \textrm{Hom}_{A}(P(i),S(j)) \to   \cdots \to \textrm{Hom}_{A}(P(n-1),S(j)) \to \textrm{Hom}_{A}(P(n),S(j)) \to 0.
\]

Note that, $\textrm{Ext}^{1}_{A}(S(i),S(j)) \cong \mathbb{F}$ for a fixed $i$  when $i<j$ in this case $S(i)$ contributes with $n-(i+1)$ to the value of $\mathcal{E}$. Now, we need to compute
\begin{align*}
    \sum_{i=2}^{n-1} n-(i+1)&=n(n-2)- \left(\frac{n(n-1)}{2}-1\right) - (n-2)\\
    &=n^2-3n+2-\frac{n^2-n-2}{2}\\
    &=\frac{2n^2-6n+4-n^2+n+2}{2}=\frac{n^2-5n+6}{2}=\frac{(n-3)(n-2)}{2}= t_{n-3}
\end{align*}

    \end{enumerate}

\noindent Adding all the cases described above $\mathcal{E}=(n-1)+(n-2)+t_{n-3}$.
\end{itemize}
\end{proof}

\noindent Now, we are able to show an explicit formula for the number of exceptional pairs of a radical square zero Nakayama algebra. For an arbitrary Nakayama algebra of Rad$^k$, finding the number all the exceptional pairs and the number of all exceptional sequences seems a little bit difficult.

\begin{thm}\label{5.3}
The number of exceptional pairs of a radical square zero Nakayama algebra is $\frac{7n^2-17n+12}{2}$.
\end{thm}

\begin{proof}
   According to Theorem \ref{Indrad} with $k=2$ the number of indecomposable modules is given by $2n-1$ and therefore the number of possible total pairs of indecomposable modules is $(2n-1)^2$. Similarly, by Theorem \ref{Hrad2} the number of pairs $\mathcal{H}$ is given by $5n-5$ and according to Theorem \ref{Hrad2}   the number of pairs $\mathcal{E}$ is given by $(n-1)+(n-2)+t_{n-3}$. Thus the number of exceptional pairs is given by: 

\begin{align*}
    \mathscr{EP}(\mathbb{A}_n/\mathrm{rad}^k)&= (2n-1)^2-(5n-5+(n-1)+(n-2)+t_{n-3})\\
    &=(2n-1)^2-(7n-8+t_{n-3})\\
    &=4n^2-4n+1-7n+8-\frac{(n-3)(n-2)}{2}\\
    &=\frac{8n^2-22n+18-n^2+5n-6}{2}\\
    &=\frac{7n^2-17n+12}{2}
\end{align*}
\end{proof}

\section{About the Integer Sequences Arising from Exceptional Pairs}

\noindent As we have seen throughout this paper, the number of exceptional pairs for any Nakayama algebra of type I can be obtained via Equation \ref{key}. This equation, in particular cases, generates numerical patterns with properties that may be interest. \par \bigskip 

\noindent The following arrays encodes number of indecomposable modules $\mathrm{Ind }Q$,  $\mathcal{H}$, $\mathcal{E}$ and  exceptional pairs for $A_{R_{(1,m)}}$. \par
\begin{table}[h!]
    \centering
\begin{tabular}{|c|c|c|c|c|c|c|c|c|}
\hline
$n/m$ & $2$ & $3$ & $4$ & $5$ & $6$ & $7$ & $8$ &$\cdots$\\
\hline
$3$ & $25$ & - & - & - & - & - & - & -\\
\hline
$4$& $64$ & $81$  & - & - & - & - & - &-\\
\hline
$5$ & $144$ & $169$ & $196$  & - & - & - & - &-\\
\hline
$6$& $289$ & $324$ & $361$ & $400$ & - & - & - &-\\
\hline
$7$ & $529$ & $576$ & $625$ & $676$ & $729$ & - & - &-\\
\hline
$8$ & $900$ & $961$ & $1024$ & $1089$ & $1156$ & $1225$ & - &-\\
\hline
$9$ & $1444$ & $1521$ & $1600$ & $1681$ & $1764$ & $1849$ & $1936$ &-\\

\hline
$\vdots $& $\vdots$ & $\vdots $ & $\vdots $ & $\vdots $ & $\vdots $ & $\vdots $ & $\vdots$& $\vdots$\\
\hline
\end{tabular}
    \caption{$|\mathrm{Ind }Q|^2$ for $A_{R_{(1,m)}}$}
    \label{Table1I}
\end{table}

\noindent The sequence in the first column is not encoded in the On-Line Encyclopedia of Integer Sequences whereas it can be obtain by the expression $\frac{(n^2 - n + 4)^2}{4}$. This expression also allow us to obtain recursively any other column $i\geq 2$ in the array via the following formula: $$\frac{(n^2-n+4)^2}{4}+(i-1)(n^2-n+5)+\sum_{h=2}^{i}2(h-2).$$

\noindent For the number of modules $\mathcal{H}$ and $\mathcal{E}$ we have the following arrays:

\begin{table}[h!]

\begin{tabular}{|c|c|c|c|c|c|c|c|c|}
\hline
$n/m$ & $2$ & $3$ & $4$ & $5$ & $6$ & $7$ & $8$ &$\cdots$\\
\hline
$3$ & $10$ & - & - & - & - & - & - & -\\
\hline
$4$& $21$ & $28$  & - & - & - & - & - &-\\
\hline
$5$ & $42$ & $51$ & $61$  & - & - & - & - &-\\
\hline
$6$& $78$ & $89$ & $102$ & $115$ & - & - & - &-\\
\hline
$7$ & $135$ & $148$ & $164$ & $181$ & $197$ & - & - &-\\
\hline
$8$ & $220$ & $235$ & $254$ & $275$ & $296$ & $315$ & - &-\\
\hline
$9$ & $341$ & $358$ & $380$ & $405$ & $431$ & $456$ & $478$ &-\\
\hline
$\vdots $& $\vdots$ & $\vdots $ & $\vdots $ & $\vdots $ & $\vdots $ & $\vdots $ & $\vdots$& $\vdots$\\
\hline
\end{tabular}
\caption{Number of pairs $\mathcal{H}$}
    \label{Table2H}
\end{table}

\begin{table}[h!]

\begin{tabular}{|c|c|c|c|c|c|c|c|c|}
\hline
$n/m$ & $2$ & $3$ & $4$ & $5$ & $6$ & $7$ & $8$ &$\cdots$\\
\hline
$3$ & $3$ & - & - & - & - & - & - & -\\
\hline
$4$& $8$ & $12$  & - & - & - & - & - &-\\
\hline
$5$ & $19$ & $25$ & $31$  & - & - & - & - &-\\
\hline
$6$& $40$ & $48$ & $57$ & $65$ & - & - & - &-\\
\hline
$7$ & $76$ & $86$ & $98$ & $110$ & $120$ & - & - &-\\
\hline
$8$ & $133$ & $145$ & $160$ & $176$ & $191$ & $203$ & - &-\\
\hline
$9$ & $218$ & $232$ & $250$ & $270$ & $290$ & $308$ & $322$ &-\\
\hline
$\vdots $& $\vdots$ & $\vdots $ & $\vdots $ & $\vdots $ & $\vdots $ & $\vdots $ & $\vdots$& $\vdots$\\
\hline
\end{tabular}
\caption{Number of pairs $\mathcal{E}$}
    \label{Table3E}
\end{table}

\noindent The sequence in the first column of the Table \ref{Table2H} is encoded in the On-Line Encyclopedia of Integer Sequences (OEIS) by A$014626$. Unlike the last case we do not have a recursively formula for any other entry of this array. In Table \ref{Table3E} the first column is not encoded in the (OEIS) as well any other column in the array.  Finally we can construct the array associated to the exceptional pairs for  $A_{R_{(1,m)}}$ 

\begin{table}[h!]

\begin{tabular}{|c|c|c|c|c|c|c|c|c|}
\hline
$n/m$ & $2$ & $3$ & $4$ & $5$ & $6$ & $7$ & $8$ &$\cdots$\\
\hline
$3$ & $12$ & - & - & - & - & - & - & -\\
\hline
$4$& $35$ & $41$  & - & - & - & - & - &-\\
\hline
$5$ & $83$ & $93$ & $104$  & - & - & - & - &-\\
\hline
$6$& $171$ & $187$ & $202$ & $220$ & - & - & - &-\\
\hline
$7$ & $318$ & $342$ & $363$ & $385$ & $412$ & - & - &-\\
\hline
$8$ & $547$ & $581$ & $610$ & $638$ & $669$ & $707$ & - &-\\
\hline
$9$ & $885$ & $931$ & $970$ & $1006$ & $1043$ & $1085$ & $1136$ &-\\
\hline
$\vdots $& $\vdots$ & $\vdots $ & $\vdots $ & $\vdots $ & $\vdots $ & $\vdots $ & $\vdots$& $\vdots$\\
\hline
\end{tabular}
\caption{Number of pairs in $A_{R_{(1,m)}}$}
    \label{Table4EP}
\end{table}

\noindent As in the case of the Table \ref{Table1I} none of the sequences of the columns are in the OEIS, however the array can be construct just using the first column which satisfy the following formula: 
\[\frac{n^4}{6}- \frac{n^3}{2} + \frac{7 n^2}{3}-4n+3\]

\noindent For the second column which count the number of exceptional pairs of  $A_{R_{(1,3)}}$ \footnote{This is just the sequence generated by the number of exceptional pairs of  $A_{R_{(1,2)}}$ plus the elements of the sequence A$027689$ in the OEIS} the formula is  

\[\frac{n^4}{6}- \frac{n^3}{2} + \frac{7 n^2}{3}-4n+3+(n-3)^2+(n-3)+4=\frac{n^4}{6} - \frac{n^3}{2}+ \frac{7 n^2}{3}-3 n + (n - 3)^2 + 4\]

\noindent For any other column with $m>3$ we have that the number of exceptional pairs of $A_{R_{(1,m)}}$ is equal to the number of exceptional pairs of $A_{R_{(1,3)}}$ plus $\displaystyle \sum_{t=3}^{m-1}t^2+\sum_{i=n-m}^{n-4}i(i+1)$.

\begin{rem}
The number of exceptional pairs change depending on the vertex where the relation start. Note that no all the relations of length $k$ give the same number of exceptional pairs. For example the number of exceptional pairs of an algebra $\mathbb{A}_7$ with a relation of type $R_{(1,2)}$ is equal to $318$ whereas the same algebra with a relation of type $R_{(4,2)}$ $260$.
\end{rem}

\begin{rem}
The number of exceptional pairs of $A_{R_{(1,2)}}$ is equal to the number $A_{R_{(n-2,2)}}$.
\end{rem}

\noindent In a similar fashion, for the case presented in Section \ref{Auslanderalg}
\begin{rem}
The integer sequence determined by the number of exceptional pairs in the Auslander algebra of a radical square zero algebra of type $\mathbb{A}_n$ does not appear in the On-Line Encyclopedia of Integer Sequences (OEIS).
\end{rem}

\noindent Finally, regarding to the case studied in Section \ref{rad2} we have the following observation 
\begin{rem}
  According with Theorem \ref{5.3}  we recall that the number of exceptional pairs is encoded by sequence A140065, the number of $|\mathrm{Ind}\, Q|^2$ by sequence A016754 (see Theorem \ref{Indrad}), and $\mathcal{H}$ by sequence A014626 in the On-Line Encyclopedia of Integer Sequences (see Theorem \ref{Hrad2}), whereas $\mathcal{E}$ corresponds to a triangular number.
\end{rem}

\subsection*{Acknowledgment}
\noindent The first author expresses his deep gratitude to his family and his wife for their constant support, patience, and encouragement throughout the development of this work. He also wishes to dedicate this paper to Professor Agustín Moreno Cañadas on the occasion of his 65th birthday.\par \bigskip 

\noindent The second author acknowledges the institutional support of Universidad de Antioquia (UdeA).

\paragraph{Author contributions} P. F. and D. R. contributed equally to the manuscript.
\paragraph{Data Availability} No datasets were generated or analysed during the current study.
\subsection*{Declarations}
\paragraph{Competing interests} The authors declare no competing interests.

% ------------------------------------------------------------------------
\end{document}